\documentclass{amsart}
\usepackage[dvipdfmx]{graphicx}

\usepackage{amssymb}
\usepackage{amsmath}
\usepackage{url}
\usepackage{enumerate}

\usepackage{hyperref}
\hypersetup{
  colorlinks  =   true,
  linkcolor   =   blue,
  filecolor   =   blue,
  urlcolor    =   blue,
  citecolor   =   blue,
}

\newcommand{\erase}[1]{}

\newtheorem{theorem}{Theorem}[section]
\newtheorem{lemma}[theorem]{Lemma}
\newtheorem{proposition}[theorem]{Proposition}

\newtheorem{_algorithm}[theorem]{Algorithm}

\newtheorem{_procedure}[theorem]{Procedure}

\newtheorem{_definition}[theorem]{Definition}
\newenvironment{definition}{\begin{_definition}\rm}{\end{_definition}}

\newtheorem{_propositiondefinition}[theorem]{Proposition-Definition}

\newtheorem{_remark}[theorem]{\it Remark}
\newenvironment{remark}{\begin{_remark}\rm}{\end{_remark}}

\newtheorem{_example}[theorem]{Example}
\newenvironment{example}{\begin{_example}\rm}{\end{_example}}

\newtheorem{_assumption}[theorem]{Assumption}

\newtheorem{_construction}[theorem]{Construction}

\newtheorem{_claim}[theorem]{Claim}

\newtheorem{_conjecture}[theorem]{Conjecture}

\numberwithin{equation}{section}
\numberwithin{table}{section}
\numberwithin{figure}{section}
\renewcommand{\qed}{\hfill {$\Box$}}

\newcommand{\C}{\mathord{\mathbb C}}
\newcommand{\DD}{\mathord{\mathbb D}}

\newcommand{\F}{\mathord{\mathbb F}}

\renewcommand{\P}{\mathord{\mathbb  P}}
\newcommand{\Q}{\mathord{\mathbb  Q}}
\newcommand{\R}{\mathord{\mathbb R}}

\newcommand{\Z}{\mathord{\mathbb Z}}

\newcommand{\DDD}{\mathord{\mathcal D}}

\newcommand{\GGG}{\mathord{\mathcal G}}

\newcommand{\LLL}{\mathord{\mathcal L}}

\newcommand{\NNN}{\mathord{\mathcal N}}
\newcommand{\OOO}{\mathord{\mathcal O}}
\newcommand{\PPP}{\mathord{\mathcal P}}

\newcommand{\RRR}{\mathord{\mathcal R}}

\newcommand{\UUU}{\mathord{\mathcal U}}

\newcommand{\SSSS}{\mathord{\mathfrak S}}

\newcommand{\mapdownsurj}{
\hbox{$\bigm\downarrow$}
\llap{\hbox{\raise 2pt\hbox{$\bigm\downarrow$}}}%
\vstrechmapdown
}

\newcommand{\mapupsurj}{
\hbox{$\bigm\uparrow$}
\llap{\hbox{\raise 2pt\hbox{$\bigm\uparrow$}}}%
\vstrechmapup
}

\newcommand{\inj}{\hookrightarrow}

\newcommand{\set}[2]{\{\,{#1}\mid {#2} \,\}}
\newcommand{\shortset}[2]{\{\, {#1} \,|  \,{#2}\,   \}}

\newcommand{\gen}[1]{\langle {#1}  \rangle}

\newcommand{\tensor}{\otimes}

\newcommand{\sprime}{\sp\prime}

\newcommand{\spar}[1]{\sp{(#1)}}
\newcommand{\spprime}{\sp{\prime\prime}}

\newcommand{\sperp}{\sp{\perp}}

\newcommand{\dual}{\sp{\vee}}

\newcommand{\semidirectproduct}{\rtimes}

\newcommand{\inv}{\sp{-1}}

\newcommand{\Hom}{\mathord{\mathrm{Hom}}}

\newcommand{\OG}{\mathord{\mathrm{O}}}

\newcommand{\id}{\mathord{\mathrm{id}}}

\newcommand{\Ker}{\operatorname{\mathrm{Ker}}\nolimits}

\newcommand{\Aut}{\operatorname{\mathrm{Aut}}\nolimits}

\newcommand{\pr}{\mathord{\mathrm{pr}}}

\newcommand{\mystrutd}[1]{\phantom{\hbox{\vrule depth #1}}}

\newcommand{\intf}[1]{\langle #1 \rangle}

\newcommand{\weyl}{\mathbf{w}}

\newcommand{\Jac}{\mathrm{Jac}}
\newcommand{\vect}[1]{\mathbf{#1}}
\newcommand{\clPPP}{\overline{\PPP}}

\newcommand{\intfLL}[1]{\intf{#1}}

\newcommand{\PP}{\mathord{\bf P}}

\newcommand{\ratmap}{\dashrightarrow}

\newcommand{\WeylGr}{W}

\newcommand{\nef}{\mathcal{N}}

\newcommand{\invol}{\iota}
\newcommand{\emb}{\epsilon}
\newcommand{\Isom}{\mathord{\mathrm{Isom}}}

\newcommand{\ChamsNY}{\mathord{\mathrm{C}\nef_Y}}
\newcommand{\innerwalls}{\mathord{\mathrm{Inn}}}
\newcommand{\outerwalls}{\mathord{\mathrm{Out}}}

\newcommand{\letterg}{\mathtt{g}}
\newcommand{\Gen}{\mathrm{Gen}}
\newcommand{\LetterGen}{\mathtt{Gen}}

\newcommand{\dgen}[1]{\langle\hskip -1pt\langle #1 \rangle\hskip -1pt\rangle}

\newcommand{\ample}{\alpha}
\newcommand{\etaH}{h_4}
\newcommand{\Indx}{\mathord{\mathrm{Indx}}}

\newcommand{\genQ}[1]{\langle{#1}\rangle_{\Q}}
\newcommand{\theL}{II_{1, 25}}

\newcommand{\Reye}{\mathrm{Rey}}
\newcommand{\rrrr}{\mathfrak{r}}

\newcommand{\PStextplot}[3]{}

\begin{document}

\title[15-nodal quartic surfaces. Part II]{15-nodal quartic surfaces. Part II: \\The automorphism group}
\author{Igor Dolgachev and Ichiro Shimada}

\address{Department of Mathematics,
University of Michigan,
2072 East Hall,
525 East University Avenue,
Ann Arbor, MI 48109-1109 USA}
\email{idolga@umich.edu}

\address{Department of Mathematics,
Graduate School of Science,
Hiroshima University,
1-3-1 Kagamiyama,
Higashi-Hiroshima,
739-8526 JAPAN}
\email{ichiro-shimada@hiroshima-u.ac.jp}

\thanks{The second author was supported by JSPS KAKENHI Grant Number 15H05738, ~16H03926,  and~16K13749}

\begin{abstract}
We describe a set of generators and defining relations for the group of birational automorphisms of a general $15$-nodal 
quartic surface in the complex projective $3$-dimensional space.
\end{abstract}

\maketitle

%
%
\section{Introduction}\label{sec:Introduction}
A quartic surface  in $\PP^3$ with $16$ ordinary double points
as its only singularities is classically known as
a Kummer quartic surface,
and has been intensively
investigated since  19th century~(see, for example,~Hudson~\cite{Hudson1905} or 
Baker~\cite{Baker1925IV}).
For a positive integer $n\le 16$,
we denote by $X_n$ a quartic surface in $\PP^3$ 
with $n$ ordinary double points (nodes)
satisfying  the following assumptions:
\begin{enumerate}[(i)]
\item
The quartic surface $X_n$ can be degenerated by acquiring additional $16-n$ nodes and hence becomes isomorphic to a Kummer quartic surface.
\item The Picard lattice $S_n$ of the minimal resolution $Y_n$ of $X_n$ is embeddable into the Picard lattice of a general Kummer surface by specialization, and it
is generated over $\Q$ by the class $h_4\in S_n$ of a plane section of $X_n$
and the classes of the exceptional curves of the resolution $Y_n\to X_n$.
In particular, the rank of $S_n$ is $n+1$.
\item
The only isometries of the transcendental lattice $T_n$ of $Y_n$ that preserve
the subspace $H^{2,0}(Y_n)$ of $T_n\tensor\C$ are $\pm 1$.
\end{enumerate}
The group $\Aut(Y_{16})$ of birational automorphisms of
a general  Kummer quartic surface $X_{16}$ was 
discribed by Kondo~\cite{Kondo1998}.
We discribe the automorphism group $\Aut(Y_{15})$ of the $K3$ surface $Y_{15}$
by the embedding of lattices $S_{15}\inj S_{16}$
induced by the specialization of $X_{15}$ to $X_{16}$.
As was proved in~\cite{PartI},  
the $15$-nodal quartic surfaces satisfy Condition~(i) 
and form an irreducible family.
If we choose a general member of  this family, then 
Conditions~(ii)~and~(iii) are satisfied.
We give a generating set of $\Aut(Y_{15})$ explicitly in Theorem~\ref{thm:15gens},
and describe the defining relations of $\Aut(Y_{15})$
with respect to this generating set
in Theorem~\ref{thm:15rels}.
\par
Our main tool is Borcherds' method~(\cite{Bor1},~\cite{Bor2}),
which was also used in the calculation of $\Aut(Y_{16})$ by Kondo~\cite{Kondo1998}.
We also compute  the defining relations
from  the tessellation of the nef-and-big cone calculated by Borcherds' method.
Our method is heavily computational, and is based on
machine-aided calculations carried out by~{\tt GAP}~\cite{GAP}.
Explicit numerical data is available
from the second author's web-page~\cite{compdata15nodal}.
\par
We hope that,
by generalizing $X_{15}$ further to $X_n$ with $n\le 14$
and looking at the embedding $S_n\inj S_{15}$,
we can calculate $\Aut(Y_n)$ for $n\le 14$.
See~\cite{AutEMS} for an example of an analysis of the change
that the automorphism group undergoes
when a $K3$ surface is generalized/specialized.
\par
This paper is organized as follows.
In Section~\ref{sec:Preliminaries},
we set up notation and terminology about lattices,
and present some computational tools
that are used throughout this paper.
In Section~\ref{sec:Borcherds}, we review Borcherds' method, and
a method  to calculate  the defining relations.
In Section~\ref{sec:AutY16},
we review the result of Kondo~\cite{Kondo1998} and Ohashi~\cite{Ohashi2009}
on the birational automorphism group $\Aut(Y_{16})$ of
a general Kummer quartic surface.
Then in Section~\ref{sec:AutY15},
we calculate the automorphism group $\Aut (Y_{15})$.
The results are given in Theorems~\ref{thm:15gens}~and~\ref{thm:15rels}.
\par 
Thanks are due to the referee for many comments.
\section{Preliminaries}\label{sec:Preliminaries}
\subsection{Lattice}\label{subsec:Lattices}
A \emph{lattice} is a free $\Z$-module $L$ of finite rank with a non-degenerate
symmetric bilinear form $\intf{\,,\,}\colon L\times L\to \Z$.
The \emph{signature} of a  lattice $L$ is
the signature of the real quadratic space $L\tensor\R$.
We say that a lattice $L$ of rank $n>1$ is \emph{hyperbolic}
if the signature of $L$ is $(1, n-1)$.
A lattice $L$ is \emph{even} if $\intf{v, v}\in 2\Z$ holds for all $v\in L$.
The cokernel of the natural embedding $L\inj L\dual:=\Hom(L, \Z)$
induced by  $\intf{\,,\,}$ is called the \emph{discriminant group} of $L$.
We say that $L$ is \emph{unimodular}  if  the discriminant group of $L$ is trivial.
When $L$ is even, the discriminant group $L\dual/L$ of $L$ is equipped with
a natural quadratic form
\[
q_L\colon L\dual/L\to \Q/2\Z,
\]
which is called the \emph{discriminant form} of $L$.
We refer to Nikulin~\cite{Nikulin1979} for the basic properties
of discriminant forms.
\par
The group of isometries  of a lattice $L$ is denoted by $\OG(L)$. We let $\OG(L)$ act on $L$ from the \emph{right}, and write the action as $v\mapsto v^g$ for $v\in L$ and $g\in \OG(L)$.
 For a subset $A$ of $L\tensor \R$, we denote by $A^g$ the image of $A$ by the action of $g$
 (\emph{not} the fixed locus of $g$ in $A$).
A vector $r\in L$ is called a \emph{$(-2)$-vector} if $\intf{r, r}=-2$.
A $(-2)$-vector $r\in L$ gives rise to an isometry  $x\mapsto x+\intf{x, r} r$
of $L$, which is called the \emph{reflection} with respect to $r$.
Let $\WeylGr(L)$ denote the subgroup of $\OG(L)$ generated by the reflections
 with respect to all  $(-2)$-vectors.
\par
Let $L$ be an even hyperbolic lattice.
We fix a \emph{positive cone} $\PPP$ of $L$, i.e.  one of the two connected components
of $\shortset{x\in L\tensor\R}{\intf{x, x}>0}$.
We put
\[
\OG(L)'=\shortset{g\in \OG(L)}{\textrm{$g$ leaves $\PPP$ invariant}}.
\]
It is a subgroup of index $2$ in $\OG(L)$ that contains  the subgroup $W(L)$.
A \emph{standard fundamental domain} of the action of $\WeylGr(L)$ on $\PPP$ is
the closure in $\PPP$ of a connected component of the complement
\[
\PPP\;\setminus\; \bigcup \, (r)\sperp
\]
of the union of hyperplanes $(r)\sperp:=\shortset{x\in \PPP}{\intf{x, r}=0}$,
where $r$ runs through the set of all $(-2)$-vectors.
Then $\PPP$ is tessellated by standard fundamental domains of  $\WeylGr(L)$,
and $\WeylGr(L)$ acts on the set of standard fundamental domains simply-transitively.
\subsection{Computational tools}\label{subsec:CompTools}
Let $L$ be an even hyperbolic lattice with a positive cone $\PPP$.
Suppose that $v_0\in L\cap \PPP$.
Then, for any integers $a$ and $b$, we can calculate the finite set
\[
\set{v\in L}{\intf{v, v}=a, \;\intf{v, v_0}=b}.
\]
Suppose that $v_1\in L$ also belongs to $\PPP$.
Then, for any negative integer $a$, we can calculate the finite set
\[
\set{v\in L}{\intf{v, v}=a, \; \intf{v, v_0}>0,  \; \intf{v, v_1}<0}.
\]
In particular, we can calculate the set
\[
\set{r\in L}{\intf{r, r}=-2,\; \intf{r, v_0}>0,  \; \intf{r, v_1}<0}
\]
of $(-2)$-vectors $r$ such that
the hyperplane $(r)\sperp$ \emph{separates} $v_0$ and $v_1$.
See~\cite{ShimadaChar5} for the details of these algorithms.
\subsection{Geometric applications}\label{subsec:GeomCompTools}
Let $S_Y$ denote the Picard lattice of a $K3$ surface $Y$
with the intersection form $\intf{\;,\;}$.
By abuse of notation,
for divisors $D, D\sprime$ of $Y$,
we write by $\intf{D, D\sprime}$ the intersection number of the divisor classes of $D$ and $D\sprime$.
Let $\ample\in S_Y$ be an ample class of $Y$, and
let $\PPP_Y$ denote the positive cone of $S_Y\otimes \mathbb{R}$  containing $\ample$.
\subsubsection{The cone $\nef_Y$}\label{subsub:nef}
Let $\overline{\nef}_Y$ be the  nef cone in $\PPP_Y$, i.e.  the set
\[
\set{x\in  S_Y\tensor \R}{
\intf{x, C}\ge 0, \textrm{for any curve $C$}}.
\]
We put
\[
{\nef}_Y:=\overline{\nef}_Y\cap\PPP_Y.
\]
We say that $h\in \PPP_Y \cap S_Y$ is \emph{nef} if $h\in \nef_Y$.
It is well-known that  $\nef_Y$ is the standard fundamental domain of
the action of $\WeylGr(S_Y)$ on $\PPP_Y$ containing the ample class $\ample$ in its interior.
Therefore, for any $v\in S_Y\cap \PPP_Y$, we can determine
whether $v$ belongs to $\nef_Y$ or not by calculating the set
of $(-2)$-vectors $r$ such that $(r)\sperp$ separates $\ample$ and $v$.
\subsubsection{Smooth rational curves}\label{subsub:smoothrationalcurves}
Let $r\in S_Y$ be a $(-2)$-vector with $d:=\intf{r, \ample}>0$. By Riemann-Roch, $r$ is an effective divisor class.  It
 is the class of a smooth rational curve
if and only if,
for any smooth rational curve $C\sprime$
with $\intf{C\sprime, \ample}<d$,
we have $\intf{r, C\sprime}\ge 0$.
Hence, by induction on $d$,
we can calculate the set of classes of all smooth rational curves
whose degree with respect to $\ample$ is $\le d$.
In particular,
we can determine whether a given $(-2)$-vector
is the class of a smooth rational curve or not.
\subsubsection{Double-plane involutions}\label{subsub:doubleplaneinvolutions}
By abuse of notation, for $h\in S_Y$,
we denote by $|h|$ the complete linear system
of a line bundle $\LLL$ with $c_1(\LLL) = h$.
A \emph{double-plane cover} 
is
a generically finite morphism $f:Y\to \PP^2$ of degree $2$.
Obviously, 
the class $h_2 := c_1(f^*(\OOO_{\PP^2}(1))$ is nef and satisfies $\intf{h_2,h_2} = 2$.
\par
Conversely, let $h_2\in S_Y$ be a nef class with $\intf{h_2, h_2}=2$.
By~\cite[Proposition~0.1]{Nikulin1991},
we can determine whether
$|h_2|$ is fixed-component free or not
by calculating the set of all vectors $v\in S_Y$ such that
$\intf{v, v}=0$ and $\intf{v, h_2}=1$.
Suppose that $|h_2|$ is fixed-component free.
Then $|h_2|$ is base-point free by~\cite[Corollary 3.2]{SaintDonat1974},
and hence $|h_2|$
defines  a  double-plane cover $\Phi\colon Y\to \PP^2$.
We denote by $\invol(h_2)\in \Aut(Y)$ the involution induced
by the deck-transformation of
$\Phi\colon Y\to \PP^2$,
and call it the \emph{double-plane involution} associated with $h_2$.
We can calculate the classes of smooth rational curves contracted by
$\Phi\colon Y\to \PP^2$,
which form an $\mathrm{ADE}$-configuration of $(-2)$-vectors.
The $\mathrm{ADE}$-type of this configuration is the
$\mathrm{ADE}$-type of the singular points of the branch curve of
$\Phi\colon Y\to \PP^2$.
The matrix representation of the action $\invol(h_2)^*$ of $\invol(h_2)$ on $S_Y$
is then calculated from this set of classes of smooth rational curves contracted by
$\Phi$.
The details of these algorithms
are given in~\cite{ShimadaChar5}. 
\section{Borcherds' method}\label{sec:Borcherds}
\subsection{Terminology}\label{subsec:Terminologies}
We fix terminology about
Borcherds' method~(\cite{Bor1},~\cite{Bor2}).
See Chapter~10~of~\cite{CSBook} for the Leech lattice and the Golay code.
See~\cite{Shimada2015} for the computational details of Borcherds' method.
\par 
We put
\[
\Omega:=\P^1(\F_{23})=\{\infty, 0, 1, \dots, 22\},
\]
and let $\F_2^{\Omega}$ and $\Z^{\Omega}$ be the modules of $\F_2$-valued
and $\Z$-valued functions on $\Omega$, respectively.
Then $\F_2^{\Omega}$ can be identified with the power set of $\Omega$
with the addition being the symmetric-difference of subsets.
We have the Golay code $\GGG$ in $\F_2^{\Omega}$.
Let $\GGG(8)$ denote the set of words of Hamming weight $8$ in $\GGG$.
Elements of $\GGG(8)$ are called \emph{octads}.
For a subset $\Sigma$ of $\Omega$, let $\nu_{\Sigma}$ denote the vector of $\Z^{\Omega}$
such that $\nu_{\Sigma}(s)=1$ if $s\in \Sigma$ and $\nu_{\Sigma}(s)=0$ otherwise.
We equip $\Z^{\Omega}$ with the negative-definite inner-product by
\[
(\vect{x}, \vect{y})\mapsto -(x_{\infty} y_{\infty}+x_{0} y_{0} +x_{1} y_{1} +\cdots+x_{22}y_{22})/8.
\]
Then the \emph{negative-definite Leech lattice} $\Lambda$ is generated in $\Z^{\Omega}$
by $2 \nu_{K}$, where $K$ runs through the set $\GGG(8)$ of octads,
and $\nu_{\Omega}-4\nu_{\{\infty\}}$.
We consider the orthogonal direct sum
\[
\theL:=U\oplus \Lambda,
\]
where $U$ is the hyperbolic plane with a fixed basis
such that the Gram matrix is  $\left[\begin{array}{cc}0&1\\1&0\end{array}\right]$. Then $\theL$ is an even unimodular hyperbolic lattice 
of rank $26$.
By Milnor's Theorem, it is unique up to
isomorphism.
A vector of $\theL$ is written as $(a,b, \lambda)$,
where $(a, b)\in U$ and $\lambda\in \Lambda$.
Let $\PPP(\theL)$ be the positive cone
containing $(1,1,0)$, and $\clPPP(\theL)$ the closure of $\PPP(\theL)$ in $\theL\tensor\R$.
A vector $\weyl$ of $\theL$ is called a \emph{Weyl vector} if $\weyl$ is non-zero, primitive in $\theL$,
contained in $\clPPP(\theL)$,
satisfying $\intfLL{\weyl, \weyl} =0$, and
such that $(\Z\weyl)\sperp/\Z\weyl$ is  isomorphic to $\Lambda$,
where $(\Z\weyl)\sperp:=\shortset{v\in \theL}{\intfLL{v, \weyl}=0}$.
A $(-2)$-vector $r$ of $\theL$ is said to be a \emph{Leech root with respect to $\weyl$}
if $\intfLL{r, \weyl}=1$.
We put
\[
\DDD(\weyl):=\set{x\in \PPP(\theL)}{\intfLL{x, r}\ge 0\;\;\textrm{for all Leech roots $r$ with respect to $\weyl$}},
\]
and call it the \emph{Conway chamber} associated with $\weyl$.
Then we have the following:
\begin{theorem}[Conway~\cite{Conway1983}]\label{thm:Conway}
The mapping $\weyl\mapsto \DDD(\weyl)$ is a bijection from the set of Weyl vectors
to the set of standard fundamental domains of the action of $\WeylGr(\theL)$ on  $\PPP(\theL)$.
\qed
\end{theorem}
\begin{example}\label{example:weyl0}
We consider the  Weyl vector
$\weyl_0:=(1,0,0)$.
For $\lambda\in \Lambda$, we put
\[
r_0(\lambda):=(-1-|\lambda|^2/2, 1, \lambda),
\]
where $|\lambda|^2$ is the square-norm of $\lambda$ in $\Lambda$.
Then the mapping
$\lambda\mapsto r_0(\lambda)$ gives a bijection from  $\Lambda$
to the set of  Leech roots  with respect to $\weyl_0$.
\end{example}
Let $S$ be an even hyperbolic lattice.
Suppose that $S$ admits a primitive embedding
\[
\emb_S\colon S\inj \theL,
\]
and let $\PPP(S)$ be the positive cone of $S$ that is mapped into $\PPP(\theL)$.
We use the same symbol $\emb_S$ to denote the 
induced embedding $\PPP(S)\inj \PPP(\theL)$.
A Conway chamber $\DDD(\weyl)$ is said to be \emph{non-degenerate}
with respect to $\emb_S$ if $\emb_S\inv (\DDD(\weyl))$
contains a non-empty open subset of $\PPP(S)$,
and when this is the case, we say that the closed subset
$\emb_S\inv (\DDD(\weyl))$ of $\PPP(S)$
is an \emph{induced chamber} or a \emph{chamber induced by $\weyl$}.
The tessellation of $\PPP(\theL)$ by Conway chambers induces
a tessellation of $\PPP(S)$ by induced chambers.
It is obvious by definition that every standard fundamental domain of
the action of $\WeylGr(S)$ on $\PPP(S)$ is also tessellated by induced chambers.
\begin{remark}\label{rem:notcongruent}
In general,
the induced chambers are not congruent to each other under the action of
$\OG(S)'$.
See~\cite{Shimada2015}~for examples.
\end{remark}
Let $G$ be a subgroup of $\OG(S)'$ such that every isometry $g\in G$ lifts to an isometry of $\theL$.
Then the tessellation of $\PPP(S)$  by  induced chambers
is preserved by the action of $G$ on $\PPP(S)$.
We consider the case  where $G$ is
\begin{equation}\label{eq:periodG}
\OG(S)^\omega:= \rho^{-1}(\{\pm \id_{S^\vee/S} \}),
\end{equation}
where
\begin{equation}\label{rho}
\rho:\OG(S)\sprime\to \OG(S^\vee/S,q_S)
\end{equation}
is defined by the natural action of $\OG(S)'$ on $S^\vee/S$. This group will be important for our geometric application.
\begin{lemma}
Every isometry of $\OG(S)^\omega$ lifts to an isometry of $\theL$.
\end{lemma}
\begin{proof}
Let $R$ be the orthogonal complement of $S$ in $\theL$.
By~Nikulin~\cite[Proposition 1.5.1]{Nikulin1979},
the even unimodular overlattice $\theL$  of the orthogonal direct sum $S\oplus R$
induces an anti-isomorphism $q_S\cong -q_R$
of discriminant forms.
For $g \in \OG(S)^\omega$, there exists an isometry $g_R$ of $R$
such that the action of $g$ on $S\dual/S$ is equal to the action of $g_R$
on $R\dual/R$ via this isomorphism $q_S\cong -q_R$.
(We can take  $\id_R$ or $-\id_R$ for $g_R$.)
Then the action of $g\oplus g_R$ on $S\oplus R$
preserves the overlattice $\theL$, and hence induces an isometry of $\theL$
whose restriction to $S$ is equal to $g$.
\end{proof}
For an induced chamber $D:=\emb_S\inv (\DDD(\weyl))$, we put
\begin{equation}\label{eq:hDprS}
\weyl_S:=\pr_S(\weyl),
\end{equation}
where $\pr_S\colon \theL\to S\dual$ is the natural projection. For $v\in S\tensor\Q$ with $\intf{v, v}<0$,
let  $(v)\sperp$ denote the real hyperplane of $\PPP(S)$ defined by $\intf{x, v}=0$.
We say that $D\cap (v)\sperp$ is a \emph{wall} of $D$
if $(v)\sperp$ is disjoint from the interior of $D$ and
$D\cap (v)\sperp$ contains a non-empty open subset of $(v)\sperp$.
We say that a vector $v\sprime\in S\tensor \Q$
\emph{defines} a wall $D\cap(v)\sperp$ of $D$
if $(v)\sperp=(v\sprime)\sperp$ and $\intf{v\sprime, x}\ge 0$ holds for all $x\in D$.
We have a unique defining vector $v\sprime\in S\dual$ of a wall $D\cap(v)\sperp$
such that $v\sprime$ is primitive in $ S\dual$,
which we call
the \emph{primitive defining vector} of the wall $D\cap(v)\sperp$.
The values
\begin{equation}\label{eq:na}
n:=\intf{v\sprime, v\sprime}, \quad a:=\intf{v\sprime, \weyl_S},
\end{equation}
where $v\sprime$ is the primitive defining vector of a wall $D\cap(v)\sperp$,
are important numerical invariants of the wall.
\par
Henceforth, we assume the following:
\begin{description}
\item[Condition~1] The orthogonal complement $R$ of  $S$ in $\theL$
can not be
embedded into the negative-definite Leech lattice $\Lambda$.
(This condition is satisfied, for example,
when $R$ contains a $(-2)$-vector.)
\end{description}
It is proved in~\cite{Shimada2015} that, under Condition~1,
each  induced chamber $D=\emb_S\inv (\DDD(\weyl))$ has only a finite number of walls,
and the primitive defining vectors of these walls can be
calculated from the Weyl vector $\weyl$ inducing $D$.
Let $D$ and $D\sprime$ be induced chambers.
We put
\[
\Isom(D, D\sprime):=\set{g\in \OG(S)}{\textrm{$g$ maps $D$ to $D\sprime$}}, \quad
\OG(S, D):=\Isom(D, D).
\]
Since we can calculate the set of walls of each induced chamber,
we can calculate all elements of $\Isom(D, D\sprime)$ and $\OG(S, D)$.
For each  wall $D\cap(v)\sperp$,
there exists a unique induced chamber $D\sprime$ such that $D\ne D\sprime$ and
$D\sprime\cap (v)\sperp=D\cap(v)\sperp$.
This induced chamber $D\sprime$ is said to be \emph{adjacent to $D$ across the wall $D\cap(v)\sperp$}.
A Weyl vector $\weyl\sprime \in \theL$ that induces
$D\sprime$ can be calculated
from the Weyl vector $\weyl$ inducing  $D$
and the primitive defining vector $v\in S\dual$ of the wall $D\cap(v)\sperp$.
See~\cite{Shimada2015} for the detail of these computations.
\subsection{Application to a $K3$ surface}\label{subsec:BorcherdsGens}
We apply the above procedure to the case
where $S$ is the Picard lattice $S_Y$ of a $K3$ surface $Y$.
Let $\PPP_Y$ be the positive cone containing an ample class,
and suppose that the primitive embedding $S_Y\inj \theL$ maps $\PPP_Y$ into $\PPP(\theL)$.
Then $\PPP_Y$ is tessellated by  induced chambers,
and the nef-and-big cone $\nef_Y$ is also tessellated by  induced chambers.
\par
For simplicity, we assume the following:
\begin{description}
\item[Condition~2] The only isometries of the transcendental lattice $T_Y$ of $Y$ that preserve
the subspace $H^{2,0}(Y)$ of $T_Y\tensor\C$ are $\pm 1$, and
the discriminant group $T_Y\dual/T_Y$ is not $2$-elementary.
\end{description}
By Torelli's theorem for algebraic $K3$ surfaces~\cite{PSS1971},
the action of $\Aut(Y)$ on $H^2(S, \Z)$ is faithful,
and an isometry $g\in \OG(S_Y)'$ is contained in the image of
the natural homomorphism
\[
\varphi_Y\colon \Aut(Y)\to \OG(S_Y)'
\]
if and only if
 $g$ preserves $\nef_Y$, and
  $g$ extends to an isometry of
the even unimodular overlattice  $H^2(S, \Z)$ of $S_Y\oplus T_Y$
that preserves $H^{2,0}(Y)$.
Suppose that $g\in \Aut(Y)$ is in the kernel of $\varphi_Y$.
Then $g$ acts on $S_Y\dual/S_Y$ as $1$,
and hence acts on $T_Y\dual/T_Y$ as $1$.
On the other hand,
since $g$ preserves $H^{2,0}(Y)$,
the action of $g$ on $T_Y$ is $\pm 1$.
The condition that
$T_Y\dual/T_Y$ is not $2$-elementary
implies that
 $-1\ne 1$ in $ \OG(T_Y\dual/T_Y)$.
Hence  $g=1$. Thus we obtain:
\begin{proposition}\label{prop:Torelli}
Suppose that Condition~2 holds.
Then the  natural homomorphism $\varphi_Y$
injective.
Moreover, an isometry $g\in \OG(S_Y)'$ is contained in the image of $\varphi_Y$
if and only if $g$ belongs to $\OG(S)^\omega$ and preserves $\nef_Y$,
where $\OG(S)^\omega$ is defined by~\eqref{eq:periodG}.
\qed
\end{proposition}
From now on, we 
regard $\Aut(Y)$ as a subgroup of $\OG(S)^\omega\subset \OG(S_Y)'$ by $\varphi_Y$.
\par
Let $\ChamsNY$ denote the set of induced chambers contained in $\nef_Y$.
Let $D$ be an element of $\ChamsNY$.
A wall $D\cap(v)\sperp$ of $D$ is said to be \emph{outer} if the hyperplane $(v)\sperp$
of $\PPP_Y$ is disjoint from the interior of $\nef_Y$, and to be \emph{inner} otherwise.
By definition, a wall $D\cap(v)\sperp$ is inner if and only if
the induced chamber $D\sprime$ adjacent to $D$ across $D\cap (v)\sperp$ is
contained in $\nef_Y$.
By definition again,
a wall $D\cap(v)\sperp$ is outer if and only if $v$ is a multiple of a $(-2)$-vector of $S_Y$,
which is the class of a smooth rational curve on $Y$.
We denote by $\innerwalls(D)$ and $\outerwalls(D)$ the set of inner walls and of outer walls of
$D$,
respectively.
\par
We further assume the following:
\begin{description}
\item[Condition~3]
We have an induced chamber $D_0\in \ChamsNY$
such that,
for every inner wall
$w=D_0\cap(v)\sperp$ of $D_0$, there exists an isometry $g_w$ in $\OG(S)^\omega$
that maps $D_0$ to the induced chamber adjacent to $D_0$ across the wall $w$.
\end{description}
Under Condition~3, induced chambers are congruent to each other.
Note that, since $w$ is inner,  the isometry $g_w\in\OG(S)^\omega$
in Condition~3
preserves $\nef_Y$, and hence $g_w\in \Aut(Y)$.
We have the following theorem, which is a special case of
a more general result given in~\cite{Shimada2015}.
\begin{theorem}\label{thm:BorcherdsMain}
Suppose that Conditions~1-3 hold.
Then the subgroup  $\Aut(Y)$ of $\OG(S_Y)^\omega$ is generated by the finite subgroup
\[
\Aut(Y, D_0):=\OG(S_Y, D_0)\cap \OG(S)^\omega
\]
and the isometries  $g_w$,
where $w$ runs through $\innerwalls(D_0)$.
\qed
\end{theorem}
%
%
%
\subsection{Defining relations}\label{subsec:BorcherdsRels}
We continue to assume that Conditions~1-3 hold.
By the classical theory of Poincar\'e relations~\cite{VinbergShvartsman},
we calculate the defining relations of $\Aut(Y)$ with respect to the generating set
given in Theorem~\ref{thm:BorcherdsMain}
assuming the following:
\begin{description}
\item[Condition~4]
The group $\Aut(Y, D_0)$ is trivial.
\end{description}
By Condition~4,
the group  $\Aut(Y)$ acts on the set $\ChamsNY$ of induced chambers in
$\nef_Y$ simply-transitively.
\begin{definition}\label{def:extraaut}
For $D\in \ChamsNY$, let $\tau_D\in \Aut(Y)$ denote the unique isometry such that
\[
D={D_0}^{\tau_D}.
\]
Let $w$ be an inner wall of $D_0$,
and $D\in C\NNN_Y$ the adjacent chamber across $w$. 
We denote the isometry $\tau_D$ by $g_w$.
\end{definition}
By Theorem~\ref{thm:BorcherdsMain}, the group  $\Aut(Y)$ is generated by
\[
\Gen:=\{\,g_w\,|\, w\in \innerwalls(D_0)\}.
\]
Let $\gen{\LetterGen}$ denote the group freely generated by
the alphabet
\[
\LetterGen:=\{\,\letterg_w \,|\, w \in \innerwalls(D_0)\}
\]
equipped with a bijection $\letterg_w\mapsto g_w$ ($w\in \innerwalls(D_0)$)
with $\Gen$.
Then the mapping
$\letterg_w\mapsto g_w$
induces a surjective
homomorphism
\[
\psi\colon \gen{\LetterGen}\to \Aut(Y).
\]
For a subset $\RRR$ of $\gen{\LetterGen}$,
we denote by $\dgen{\RRR}$ the minimal normal
subgroup of $\gen{\LetterGen}$ containing $\RRR$. 
Our goal is to find a subset $\RRR$ of $\gen{\LetterGen}$ 
such that that $\Ker\psi = \dgen{\RRR}$. 
In the following,
an element of the free group $\gen{\LetterGen}$ is written as a sequence
\[
(\letterg_{w_1}^{\pm 1}, \cdots, \letterg_{w_m}^{\pm 1} )
\]
of letters $\letterg_w^{\pm 1}$.
\par
Suppose that two induced chambers $D, D\sprime\in \ChamsNY$ are adjacent,
and let $w_{D, D\sprime}=w_{D\sprime, D}$ be the wall between them.
Then
\[
w_{D\sprime \gets D}:={(w_{D\sprime, D})}^{\tau_D\inv}
\]
is an inner wall of $D_0$ such that
$D_0\sprime:={(D\sprime)}^{ \tau_D\inv}$ is the induced chamber adjacent to $D_0$
across $w_{D\sprime \gets D}$. Therefore, putting
\[
g_{D\sprime\gets D}:=g_{w_{D\sprime \gets D}}\;\in\; \Gen,
\]
we have $D_0\sprime={D_0}^{g_{D\sprime\gets D}}$ and
we obtain
\[
\tau_{D\sprime}=g_{D\sprime\gets D}  \cdot \tau_D.
\]
In the same way, we have $\tau_D=g_{D\gets D\sprime}  \cdot \tau_{D\sprime}$,
and hence
\begin{equation}\label{eq:gDDgDD}
g_{D\sprime\gets D} \cdot g_{D\gets D\sprime} =1,
\end{equation}
from which we deduce that 
%
\begin{equation}\label{eq:invinv}
\textrm{the generating set $\Gen$ is invariant under  $g\mapsto g\inv$.}
\end{equation}
We put
\begin{equation}\label{eq:RRR1}
\RRR_1:=\set{(\letterg_w, \letterg_{w\sprime})}{g_wg_{w\sprime}=1}.
\end{equation}
%
We obviously have $\dgen{\RRR_1}\subset \Ker \psi$.
\par
A \emph{chamber path in $\ChamsNY$} is a sequence
\[
\DD=(D\spar{1}, \dots, D\spar{m})
\]
of induced chambers contained in $\nef_Y$
such that $D\spar{i-1}$ and $D\spar{i}$  are adjacent for $i=2, \dots, m$.
(For a chamber path $\DD$, we have $D\spar{i-1}\ne D\spar{i}$
for $i=2, \dots, m$, but in general,
we may have $D\spar{i}= D\spar{j}$ for some $i, j$ with $|i-j|>1$.)
When $D\spar{1}=D\spar{m}$, the chamber path $\DD$
is called a \emph{chamber loop in $\ChamsNY$}.
For a chamber loop
\[
\DD=(D_0, \dots, D_m)
\]
in $\ChamsNY$ starting with
the fixed induced chamber $D_0=D_m$,
we define a relation $R(\DD)\in \Ker\psi$ associated with $\DD$
as follows.
We put
\[
g_i:=g_{D_i\gets D_{i-1}} \in \Gen
\]
for $i=1, \dots, m$,
and let $\letterg_i \in \LetterGen$ be the letter corresponding to $g_i$.
Then we have
\[
\tau_{D_i}=g_i\cdots g_1
\]
for $i=1, \dots, m$.
Since $D_0=D_m$, we have $g_m\cdots g_1=1$ and the sequence
\[
R(\DD):=(\letterg_1\inv, \dots, \letterg_m\inv)
\]
belongs to the kernel of $\Ker \psi$.
\par
Let $D$ be an element of $\ChamsNY$.
A non-empty closed subset $f$ of $D$ is said to be a \emph{face} of $D$
if $f$ is an intersection of walls of $D$.
The \emph{dimension} $\dim f$ of a face $f$ of $D$ is
the dimension of the minimal linear subspace of $S_Y\tensor\R$
containing $f$.
The \emph{codimension} of $f$ is defined to be $\dim (S_Y\tensor\R)-\dim f$.
The faces of $D$ with  codimension $1$ are exactly the walls of $D$.
If $f$ is a face of $D$ with codimension $2$,
then there exist exactly two walls of $D$
containing $f$, and $f$ is equal to the intersection of these two walls.
A face of $D$ is said to be \emph{inner} if
 a general  point of $f$ belongs to the interior of $\nef_Y$.
\begin{remark}
Suppose that a face $f$ of $D$ with codimension $2$ is written as
$f=w\cap w\sprime$, where $w$ and $w\sprime$ are walls of $D$.
Even if $w$ and $w\sprime$ are inner,
the face $f$ may fail to be inner.
See~Figure~\ref{fig:specialinner},
in which
the black dot is a face $f=w\cap w\sprime$ of $D$
with
$w$ and $w\sprime$ being inner,  and the thick line is  bounding $\nef_Y$.
\setlength{\unitlength}{.3cm}
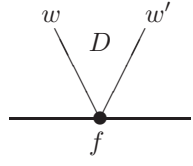
\begin{figure}
\begin{picture}(8,6)(-4,-1)
  \put(0,0){\line(1,2){2}}
  \put(0,0){\line(-1,2){2}}
\linethickness{1pt}
  \put(0,0){\line(1,0){4}}
  \put(0,0){\line(-1,0){4}}
  \put(0,0){\circle*{.6}}
  \put(-.5,3){$D$}
  \put(-2.6,4.3){$w$}
  \put(2,4.3){$w\sprime$}
  \put(-0.5,-1.4){$f$}
\end{picture}
\caption{$f$ is not an inner face}\label{fig:specialinner}
\end{figure}
\end{remark}
Let $f$ be an inner face of $D\in \ChamsNY$ with codimension $2$.
A \emph{simple chamber loop around $(f, D)$}
is a chamber loop $(D\spar{0}, \dots, D\spar{m})$ in $\ChamsNY$  with $D\spar{0}=D\spar{m}=D$
such that each $D\spar{i}$ contains $f$ as a face and that
$D\spar{i}\ne D\spar{j}$ unless $i= j$ or $\{i, j\}= \{0, m\}$.
Note that,
for a fixed $(f, D)$,
there exist exactly two simple chamber loops around $(f, D)$,
which have opposite orientations.
\par
From each inner face $f$ of $D_0$ with codimension $2$,
we choose a simple chamber loop $\DD(f, D_0)$ around $(f, D_0)$,
and make the set
\begin{equation}\label{eq:RRR2}
\RRR_2:=\shortset{R(\DD(f, D_0))}{\textrm{$f$ is an inner face of $D_0$ with codimension $2$}}.
\end{equation}
Its elements are called the \emph{Poincar\'e relations}.
%
By \cite[Chapter 2, Section 1.3, Theorem~1.3]{VinbergShvartsman}, we have the following:
\begin{theorem}\label{thm:BorcherdsRels}
Suppose that Conditions~1-4 hold.
Then the kernel of the surjective homomorphism $\psi\colon \gen{\LetterGen}\to \Aut(Y)$
is equal to $\dgen{\RRR_1\cup \RRR_2}$.
\qed
\end{theorem}
\subsection{Making the list of all faces of $D_0$ with codimension $2$}
\label{subsec:algofaces}
Let $w=D_0\cap (v)\sperp$  be a wall of $D_0$.
Then the list of faces of $D_0$
with codimension $2$ contained in $w$
can be calculated as follows.
Let $\genQ{w}$ be the minimal subspace
of $S_Y\tensor\Q$ containing $w$.
Since the intersection form of $S_Y\tensor \Q$
restricted to $\genQ{w}$ is non-degenerate,
we have the orthogonal projection
\[
\pr_w\colon S_Y\tensor\Q\to \genQ{w}.
\]
Let $\{D_0\cap (v_k)\sperp\}$ be the list of walls of $D_0$
with $v_k$ defining the wall $D_0\cap (v_k)\sperp$,
i.e., $\intf{x, v_k}\ge 0$ for all $x\in D_0$.
We construct a set $\UUU_w$ of vectors of $\genQ{w}$
with the following properties:
\begin{itemize}
\item [(a)]
If $\pr_w(v_k)\in \genQ{w}$ has a negative  square-norm,
then
$\pr_w(v_k)$ is
a positive-rational multiple of an element $u\in \UUU_w$.
\item[(b)]
Every element  $u$ of $ \UUU_w$  is
a positive-rational multiple of $\pr_w(v_k)$
for some $v_k$
with  $\intf{\pr_w(v_k), \pr_w(v_k)}<0$.
\item[(c)]
No two vectors of $\UUU_w$ are linearly dependent.
\end{itemize}
Then $w$ is defined in $\genQ{w}\tensor\R$  by 
\[
w=\set{x\in \genQ{w}\tensor\R}{\intf{x, u}\ge 0\;\;\textrm{for all}\;\; u \in \UUU_w}.
\]
Let $w\sprime=D_0\cap(v\sprime)\sperp$ be another wall of $D_0$
defined by $v\sprime$.
Then  $w\cap w\sprime$ is a face of $D_0$
with codimension $2$ if and only if the following hold:
\begin{itemize}
\item The square-norm of $\pr_w(v\sprime)$ is $<0$.
\item Let $u\sprime $ be the unique vector of $\UUU_w$
that is a rational multiple of $\pr_w(v\sprime)$.
Then the problem
\[
\textrm{\hskip 10mm ``minimize $\intf{u\sprime, x}$ under constraints $\intf{u\spprime, x}\ge 0$
for all $u\spprime \in \UUU_{w}\setminus \{u\sprime\}$"}
\]
of linear programming on $\genQ{w}$
is unbounded to $-\infty$.
\end{itemize}
\begin{remark}
This algorithm is easily generalized to an algorithm of
making the list of faces of higher codimensions.
See~\cite{ShimadaVeniani2019} for
an application  of this algorithm to the classification of Enriques involutions
on a $K3$ surface.
\end{remark}
\subsection{An algorithm to calculate the relation}
\label{subsec:algorels}
Let $f$ be a face of $D_0$ with codimension $2$.
We present an algorithm
to determine whether $f$ is inner or not,
and when $f$ is inner,
to calculate one of the two  simple chamber loop
$(D_0, \dots, D_m)$ around $f$ from $D_0$ to $D_m=D_0$
and
the relation
$g_m\cdots g_1=1$
obtained from
$(D_0, \dots, D_m)$,
where $(g_m, \dots, g_1)$ are
the sequence of elements of $\Gen$ such that
$g_i \cdots g_1=\tau_{D_i}$ for $i=1, \dots, m$.
\par
We put $i=0$ and $\tau_0=1$.
In the calculation below,
we have that
$D_i$ is an induced chamber containing $f$ as a face,
and ${\tau_i}\in \Aut(Y)$ the unique isometry that maps $D_0$ to $D_i$.
\begin{enumerate}[(i)]
\item
Let $f_i$ be the image of $f$ by $\tau_i\inv$.
Since $f$ is a face of $D_i$, we see that $f_i$ is a face of $D_0$.
\item Find the two walls $w_i\sprime$ and $w_i\spprime$
of $D_0$ such that $f_i=w_i\sprime \cap w_i\spprime$.
\item If $w_i\sprime$ or $w_i\spprime$ is outer,
then $f$ is outer, and we quit.
\item
Suppose that $w_i\sprime$ and $w_i\spprime$ are inner.
\begin{enumerate}[(a)]
\item
When  $i=0$,
we put
$g_1:=g_{w_0\sprime}$, $D_1:={D_0}^{g_1}$, $\tau_1:=g_1$.
(If we interchange $w_i\sprime$ and $w_i\spprime$,
we obtain the opposite simple chamber loop around $(f, D_0)$.)
\item
Suppose that $i>0$.
We put
$D_{i+1}\sprime:={D_0}^{g_{w_i\sprime} g_i\cdots g_1}$
and
$D_{i+1}\spprime:={D_0}^{g_{w_i\spprime} g_i\cdots g_1}$.
Then either
$ D_{i+1}\sprime=D_{i-1}$ or $ D_{i+1}\spprime=D_{i-1}$ holds.
In the former case, we put
\[
g_{i+1}:=g_{w_i\spprime}, \quad D_{i+1}:=D_{i+1}\spprime,
\]
and in the latter case, we put
\[
g_{i+1}:=g_{w_i\sprime}, \quad D_{i+1}:=D_{i+1}\sprime.
\]
We then put $\tau_{i+1}:=g_{i+1}\tau_i$.
\end{enumerate}
\item If $\tau_{i+1}=1$, then we stop
and return $(D_0, \dots, D_{i+1})$.
If $\tau_{i+1}\ne 1$,
we increment $i$
and repeat the process from (i) again.
\end{enumerate}
\section{The birational automorphism group of a $16$-nodal quartic surface}\label{sec:AutY16}
From now on,
let $X_n$ be a $n$-nodal quartic surface satisfying the assumptions (i)-(iii) at the beginning
of Introduction.
Let $S_n$ be the Picard lattice of the minimal resolution $Y_n$ of $X_n$,
and let $\PPP_n$ be the positive cone of $S_n$ containing
an ample class.
Let $\nef_n=\overline{\nef}_n\cap \PPP_n$ be the intersection of 
the nef cone of $Y_n$
with $\PPP_n$,
and  let $h_4\in S_n$ denote the class of a plane section of $X_n\subset \PP^3$.
The specialization of $X_{n}$ to $X_{16}$ gives an embedding of lattices
$S_{n}\inj S_{16}$ that maps $h_4\in S_{n}$ to $h_4\in S_{16}$. \par
\medskip
The minimal resolution $Y_{16}$ of
a general Kummer quartic surface $X_{16}$ is the Kummer surface
associated with the Jacobian variety $\Jac(C_0)$ of a general genus $2$ curve $C_0$.
A finite generating set of
the automorphism group of $Y_{16}$ was calculated by Kondo~\cite{Kondo1998}
by Borcherds' method,
and Ohashi~\cite{Ohashi2009} supplemented this result
with another set of generators.
We review their results briefly.
\par
Let $P_1, \dots, P_6$ be the Weierstrass points of  $C_0$.
We have a quotient morphism  $\Jac(C_0)\to X_{16}$
by the action of $\{\pm 1\}$ on $\Jac(C_0)$.
The $16$  nodes of $X_{16}$ correspond to
the points of
\[
\Jac(C_0)_2:=\set{x\in \Jac(C_0)}{2x=0}=\{[0]\}\cup \set{[P_i-P_j]}{1\le i<j\le 6}.
\]
Let $N_0$ and $N_{ij}$ be the smooth rational curves  on $Y_{16}$
corresponding to the points $[0]$ and $[P_i-P_j]$ of $\Jac(C_0)_2$, respectively,
which we call \emph{nodal curves} on $Y_{16}$.
Let $\theta$ be a theta characteristic of the curve $C_0$, that is,
$\theta$  is a divisor class of $C_0$ of degree $1$ 
such that $2\theta$ is linearly equivalent to
the canonical divisor of $C_0$.
Then the image of the embedding $x\mapsto [x-\theta]$ of $C_0$ into $\Jac(C_0)$
yields a \emph{trope-conic} of $X_{16}$.
Let $T_i$ and $T_{ij}$ be the smooth rational curves on $Y_{16}$
corresponding to the trope-conics obtained from
the theta characteristic $[P_i] $ ($1\le i\le 6$) and $[P_i+P_j-P_6]$ ($1\le i<j\le 5$),
respectively.
\par
It is known (\cite[Lemma 3.1]{Kondo1998}) that the classes of these $16+16$ curves $N_0$, $N_{ij}$, $T_j$, $T_{ij}$
generate the Picard lattice $S_{16}$ of $Y_{16}$.
The discriminant group $S_{16}\dual/S_{16}$ of $S_{16}$ is isomorphic to
$(\Z/2\Z)^4\oplus (\Z/4\Z)$. 
By \cite[Theorem 4.1]{Keum1997}, Condition~2 in Section~\ref{subsec:BorcherdsGens} 
and Condition~(iii) from Introduction is satisfied.
\par
It is also well-known that 
$X_{16}$ is self-dual,
that is, the dual surface $X\sprime_{16}\subset \PP^3$ is
isomorphic to $X_{16}$.
The minimal resolution $Y_{16}\to X\sprime_{16}$ contracts
the $16$ curves $T_i$, $T_{ij}$ to the  nodes,
and maps the other $16$ curves $N_0$, $N_{ij}$ to conics.
The Gauss map $X_{16}\ratmap X\sprime_{16}$ induces
an involution of $Y_{16}$ that interchanges
the $16$ curves $N_0$, $N_{ij}$ and the $16$ curves  $T_i$, $T_{ij}$.
This involution is called a \emph{switch}.
Let $h\sprime_4\in S_{16}$ be the class of a plane section of
the dual quartic surface $X_{16}\sprime\subset \PP^3$.
Then we have
\begin{equation}\label{eq:h4h4dual}
h_4\sprime=3 h_4-N_0-\sum N_{ij},
\quad
h_4=3 h_4\sprime-\sum T_{i} -\sum T_{ij}.
\end{equation}
\par
Kondo~\cite{Kondo1998} embedded $S_{16}$ into $\theL$ primitively.
Recall that, for each octad $K\in \GGG(8)$,
we have a vector $2\nu_{K}$ of the Leech lattice $\Lambda$ and
a Leech root $r_0(2\nu_{K})$ of $\theL=U\oplus \Lambda$
with respect to the Weyl vector $\weyl_0=(1,0,0)\in \theL$.
(See Example~\ref{example:weyl0}.)
With each smooth rational curve $E_k$
in  the $16+16$ curves $N_0$, $N_{ij}$, $T_i$, $T_{ij}$,
an octad $K_k$ is associated as in Table~\ref{table:octads}, and hence
a Leech root $r_k:=r_0(2\nu_{K_k})$ is also associated.
Then
the intersection number of $E_k$ and $ E_{k\sprime}$ on $Y_{16}$
is equal to the intersection number
of $r_k$ and $r_{k\sprime}$ in $\theL$
for all $k, k\sprime=1, \dots, 32$,
and thus  we obtain an embedding
\[
\emb_{16}\colon S_{16}\inj \theL,
\]
which turns out to be primitive.
\begin{table}
{\footnotesize
\[
\begin{array}{cclccl}
N_{0} &   \colon & \{\infty, 0, 1, 7, 12, 13, 14, 20\}, &
N_{12} &  \colon & \{\infty, 0, 1, 13, 15, 17, 18, 19\}, \\
N_{13} &  \colon & \{\infty, 0, 1, 6, 11, 14, 15, 16\}, &
N_{14} &  \colon & \{\infty, 0, 1, 5, 10, 12, 16, 19\}, \\
N_{15} &  \colon & \{\infty, 0, 1, 7, 10, 11, 17, 22\}, &
N_{16} &  \colon & \{\infty, 0, 1, 5, 6, 18, 20, 22\}, \\
N_{23} &  \colon & \{\infty, 0, 1, 8, 16, 17, 20, 21\}, &
N_{24} &  \colon & \{\infty, 0, 1, 3, 7, 9, 16, 18\}, \\
N_{25} &  \colon & \{\infty, 0, 1, 8, 9, 14, 19, 22\}, &
N_{26} &  \colon & \{\infty, 0, 1, 3, 12, 15, 21, 22\}, \\
N_{34} &  \colon & \{\infty, 0, 1, 5, 9, 11, 13, 21\}, &
N_{35} &  \colon & \{\infty, 0, 1, 4, 6, 9, 12, 17\}, \\
N_{36} &  \colon & \{\infty, 0, 1, 4, 5, 7, 8, 15\}, &
N_{45} &  \colon & \{\infty, 0, 1, 3, 4, 11, 19, 20\}, \\
N_{46} &  \colon & \{\infty, 0, 1, 4, 10, 14, 18, 21\}, &
N_{56} &  \colon & \{\infty, 0, 1, 3, 6, 8, 10, 13\}. \\
\end{array}
\]
\vskip -.0cm
\[
\begin{array}{cclccl}
T_{1} & \colon & \{\infty, 0, 2, 3, 4, 8, 9, 21\}, &
T_{2} & \colon & \{\infty, 0, 2, 4, 5, 6, 10, 11\}, \\
T_{3} & \colon & \{\infty, 0, 2, 3, 10, 18, 19, 22\}, &
T_{4} &  \colon & \{\infty, 0, 2, 6, 8, 15, 17, 22\}, \\
T_{5} &  \colon & \{\infty, 0, 2, 5, 15, 16, 18, 21\}, &
T_{6} &  \colon & \{\infty, 0, 2, 9, 11, 16, 17, 19\}, \\
T_{12} &  \colon & \{\infty, 0, 2, 7, 8, 10, 14, 16\}, &
T_{13} & \colon & \{\infty, 0, 2, 10, 12, 13, 17, 21\}, \\
T_{14} & \colon & \{\infty, 0, 2, 3, 7, 11, 13, 15\}, &
T_{15} &  \colon & \{\infty, 0, 2, 4, 12, 14, 15, 19\}, \\
T_{23} &  \colon & \{\infty, 0, 2, 6, 9, 13, 14, 18\}, &
T_{24} &  \colon & \{\infty, 0, 2, 5, 8, 13, 19, 20\}, \\
T_{25} &  \colon & \{\infty, 0, 2, 4, 7, 17, 18, 20\}, &
T_{34} &  \colon & \{\infty, 0, 2, 3, 6, 12, 16, 20\}, \\
T_{35} &  \colon & \{\infty, 0, 2, 11, 14, 20, 21, 22\}, &
T_{45} &  \colon & \{\infty, 0, 2, 5, 7, 9, 12, 22\}. \\
\end{array}
\]
}
\caption{$16+16$ Octads}\label{table:octads}
\end{table}
This primitive embedding $\emb_{16}$ has the following properties.
\par
\medskip
(1) Let $R_{16}$ denote the orthogonal complement of
$S_{16}$ in $\theL$.
Then $R_{16}$ is negative-definite of rank $9$ and
contains $(-2)$-vectors that form
the Dynkin diagram of type $6A_1+A_3$.
In particular, Condition~1 is satisfied.
\par
(2) The Conway chamber associated with the Weyl Vector
$\weyl_0=(1,0,0)$ is non-degenerate with respect to $\emb_{16}$,
and the induced chamber $D_{16}:=\emb_{16}\inv(\DDD(\weyl_0))$
is contained in the nef-and-big cone $\nef_{16}$.
The vector
\[
\ample_{16}:=\pr_S(\weyl_0)\in S_{16}\dual
\]
is in fact a vector of $S_{16}$,
sits in the interior of $D_{16}$, 
and is a very ample class of degree $8$ that embeds $Y_{16}$
into $\PP^5$ as a $(2,2,2)$-complete intersection 
\begin{equation}\label{eq:222eqs}
\setlength\arraycolsep{1pt}
\renewcommand{\arraystretch}{1.2}
\begin{array}{rcrcrcrcrcrcccc}
x_1^2 & + & x_2^2 & + &  x_3^2 & + & x_4^2 & + &  x_5^2 & + & x_6^2 & \;&= &\; 0,  \\
\lambda_1 x_1^2 & + & \lambda_2 x_2^2 & + &  \lambda_3 x_3^2 & + &  \lambda_4 x_4^2 & + &  \lambda_5 x_5^2 & + &  \lambda_6 x_6^2 & \;&= &\; 0,  \\
\lambda_1^2 x_1^2 & + & \lambda_2^2 x_2^2 & + &  \lambda_3^2 x_3^2 & + &  \lambda_4^2 x_4^2 & + &  \lambda_5^2 x_5^2 & + &  \lambda_6^2 x_6^2 & \;&= &\; 0,
\end{array}
\end{equation}
where $\lambda_1, \dots, \lambda_6$
are complex numbers such that
the genus $2$ curve $C_0$ is defined by
\[
w^2=(t-\lambda_1)(t-\lambda_2)\cdots (t-\lambda_6).
\]
(See Baker~\cite[Chapter 7]{Baker1925IV}, Hudson~\cite[\S 31]{Hudson1905} 
and modern expositions in Shioda~\cite{Shioda1977} or Dolgachev~\cite[10.3.3]{CAG}).  
The $16+16$ curves $N_{0},N_{ij}, T_i,T_{ij}$
are mapped to the lines of this $(2,2,2)$-complete intersection.
The group $\OG(S_{16}, D_{16})$ is equal to the stabilizer subgroup of
$\ample_{16}$ in $\OG(S_{16})\sprime$ and,
by~\cite[Lemma 4.5]{Kondo1998},  we have 
\begin{equation}\label{kondosym}
\OG(S_{16}, D_{16})\cong  (\Z/2\Z)^5\semidirectproduct \SSSS_6.
\end{equation}
The group $\OG(S_{16}, D_{16})\cap  \OG(S_{16})^\omega$ is isomorphic to 
$(\Z/2\Z)^5$,
where $\OG(S_{16})^\omega$ is defined by~\eqref{eq:periodG}.
Note that $\Aut(Y_{16}, D_{16})=\OG(S_{16}, D_{16})\cap  \OG(S_{16})^\omega$ 
is equal to the projective automorphism group
$\Aut(Y_{16}, \ample_{16})\cong (\Z/2\Z)^5$ of
the $(2,2,2)$-complete intersection $Y_{16}\subset \PP^5$ 
that consists of switching signs of the coordinates. 
Note that the six involutions corresponding to the switch of the sign at one of the coordinates are the involutions arising from one of six realizations of the Kummer surface as the focal surface of a congruence of lines of bidegree $(2,2)$. 
One of them acts on $N_0, N_{ij}, T_{i}, T_{ij}$ as follows:
\[
N_0\leftrightarrow T_6,
\quad
N_{i6} \leftrightarrow T_{i},
\quad
N_{ij} \leftrightarrow T_{ij} \quad(1\le i<j<6).
\]
\par
(3) The walls of $D_{16}$ are as in Table~\ref{table:WallsD16}.
The $32$ outer walls are defined by the classes of the curves $N_0, N_{ij}, T_{i}, T_{ij}$.
The action of  $\OG(S_{16}, D_{16})$ decomposes the set of walls into four orbits,
and each orbit is further decomposed into smaller orbits by the action of the subgroup
$\Aut(Y_{16}, \ample_{16})\cong (\Z/2\Z)^5$ of $\OG(S_{16}, D_{16})$
as indicated in the third column of Table~\ref{table:WallsD16}.
(For example,
the large orbit No.~2 of size $60$ is decomposed into $15$ small orbits,
each  of which is of size $4$.)
For each wall,
the values $n=\intf{v, v}$, $a=\intf{v, \ample_{16}}$ and
$d=\intf{\ample_{16}, \weyl\sprime_S}$
are also given in Table~\ref{table:WallsD16},
where $v$ is the primitive defining vector of the wall $D_{16}\cap (v)\sperp$,
$D\sprime$ is the induced chamber adjacent to $D_{16}$ across the wall,
$\weyl\sprime$ is
a Weyl vector inducing $D\sprime$,
and $\weyl\sprime_S=\pr_S(\weyl\sprime)$ is defined in~\eqref{eq:hDprS}.
\par
(4) For each inner wall $w$, there exists an involution $g_w$
that maps $D_{16}$ to the induced chamber adjacent to $D_{16}$ across the wall $w$.
For an  inner wall in the  orbit No.~2,
this involution is the Hutchinson-G\"opel involution,
which is an Enriques involution.
For an inner wall in the   orbit No.~3,
this involution is  obtained
by  the projection $X_{16}\ratmap \PP^2$
from a node of $X_{16}$ or
by the  projection $X\sprime_{16}\ratmap \PP^2$
from a node of $X\sprime_{16}$.
For an inner wall in the   orbit No.~4,
this involution is the Hutchinson-Weber involution,
which is again an Enriques involution.
(See Hutchinson~\cite{Hutchinson1899},~\cite{Hutchinson1900}, ~\cite{Hutchinson1901} for
the Hutchinson-G\"opel and the Hutchinson-Weber involutions.)
\begin{table}
\begin{equation}
\begin{array}{ccccccc}
\textrm{No.}&\textrm{type}&&\textrm{size}&n&a&d\\
\hline
1&\textrm{outer}&\textrm{$(-2)$-curve}&1\times32&-2&1&9\\
\hline
2&\textrm{inner}&\textrm{G\"opel}&15\times4&-1&2&16\\
3&\textrm{inner}&\textrm{Projection}&1\times32&-1&3&26\\
4&\textrm{inner}&\textrm{Weber}&6\times32&-3/4&3&32
\end{array}
\end{equation}
\caption{Walls of $D_{16}$}\label{table:WallsD16}
\end{table}
\par
\medskip
By these results and Theorem~\ref{thm:BorcherdsMain}, we obtain the following:
\begin{theorem}[Kondo~\cite{Kondo1998}, Ohashi~\cite{Ohashi2009}]
The automorphism group $\Aut(Y_{16})$ of a general Jacobian Kummer surface  $Y_{16}$
is generated by the projective automorphism
$\Aut(Y_{16}, \ample_{16})$ of
the $(2,2,2)$-complete intersection model, 
the involutions obtained from the  projections with the center
being the  nodes of the quartic surface model $X_{16}$ or its dual $X\sprime_{16}$,
the Hutchinson-G\"opel involutions, and 
the Hutchinson-Weber involutions.
\qed
\label{thm:AutY16}
\end{theorem}
\begin{remark}
Kondo~\cite{Kondo1998} used Keum's automorphisms~\cite{Keum1997}
as a part of the generating set of $\Aut(Y_{16})$.
Ohashi~\cite{Ohashi2009} showed that Keum's automorphisms
can be replaced by Hutchinson-Weber involutions.
\end{remark}
\begin{remark}\label{rem:smoothrationalcurvesonY16}
%
%
By the method in Section~\ref{subsub:smoothrationalcurves},
we calculate the sets of classes of all smooth rational curves
$C$ on $Y_{16}$ with $\intf{C, \ample_{16}}=d$ for
$d=1, \dots, 14$.
The sizes of these sets are as follows:
\[
\begin{array}{c|cccccccccccccccc}
d & 1 & 2 & 3 & 4 & 5 & 6 & 7 & 8 & 9 & 10 & 11 & 12 & 13 & 14  \\
\hline
\textrm{size} & 32& 0& 0& 0& 480& 0& 320& 0& 15264& 0& 1920& 0& 120992& 0 \rlap{\;.}
\end{array}
\]
\end{remark}
\section{The birational automorphism group of a $15$-nodal quartic surface}\label{sec:AutY15}
Let $X_{15}$ be a $15$-nodal quartic surface.
It is proven in~\cite{PartI} that it is isomorphic to a hyperplane section of the fixed hypersurface in $\PP^4$ isomorphic to the Castelnuovo-Richmond-Igusa quartic hypersurface $\textrm{CR}_4$ with $15$ double lines. When the hyperplane specializes to a tangent hyperplane, the section acquires an additional node $p_0$ and becomes isomorphic to $X_{16}$. This proves that any $X_{15}$ is obtained by  smoothing one node of $X_{16}$, and hence it satisfies condition (i) from Introduction. 

Fix one of the six possible  realizations of $X_{15}$ as the focal surface of
a congruence of lines in $\PP^3$ of order $2$ and class $3$. Each such realization comes with an involution $\sigma^{(i)}$ of $Y_{15}$ whose quotient is a quintic del Pezzo surface $\mathsf{D}$. When $X_{15}$ specializes to $X_{16}$, the involution acquires the new node as its fixed point, and the quotient of $Y_{16}$ by the lift of this involution becomes isomorphic to a quartic del Pezzo surface, the blow-up of one point on 
$\mathsf{D}$. 
Following~\cite{PartI}, we can take one of the involutions $\sigma$ among the involutions $\sigma^{(i)}$ in such a way that the nodes of $X_{15}$ are indexed by $2$-elements  subsets  of $[1,6]=\{1, \dots, 6\}$ with nodal curves $E_{ij}$ such that $\sigma(E_{ij})$ are proper transforms of trope-conics for $i,j\ne 6$ and $\sigma(E_{i6})$ is the proper transform of a trope-quartic curve.  
The embedding 
$$ \emb_{15,16}\colon S_{15}\inj S_{16}$$
is defined by mapping
\begin{eqnarray*}
\emb_{15,16}(E_{ij}) &=& N_{ij} \;\;\; (1\le i<j\le 6),\\
\emb_{15,16}(\sigma(E_{ij})) &=& T_{ij} \;\;\;\;(1\le i<j\le 5),\\
\emb_{15,16}(\sigma(E_{i6})) &=& T_{i}+N_0+T_6, \\
\emb_{15,16}(h_{4})& =& h_{4}.
\end{eqnarray*}
In view of this notation, it is natural to identify $\sigma$ with $\sigma^{(6)}$. 
Other involutions $\sigma^{(\nu)}$ and the corresponding embedding 
$\emb_{15,16}^{(\nu)}$ are defined by applying 
a permutation from $\mathfrak{S}_6$ that sends $6$ to $\nu$.
It was proved in~\cite{PartI} that
the embedding $\emb_{15, 16}$ is primitive and
\[
\emb_{15,16}(S_{15}) = (\mathbb{Z}N_0)^\perp.
\]
\par
The minimal resolutions of $15$-nodal quartics  form an open subset of
the coarse moduli space of $K3$ surfaces lattice-polarized by $S_{15}$.
If $Y_{15}$ is chosen generally in this moduli space,
then $Y_{15}$  satisfies Conditions~(ii)~and~(iii) at the beginning of  Introduction.
\par
Henceforth
we regard $S_{15}$ as a primitive sublattice of $S_{16}$ embedded by  $\emb_{15, 16}$.
The discriminant group of $S_{15}$ is isomorphic to
$(\Z/2\Z)^5\oplus (\Z/4\Z)$, which,
combined with Condition~(iii),  implies that Condition~2
in Section~\ref{subsec:BorcherdsGens} is satisfied. The positive cone $\PPP_{15}$ of $S_{15}$ containing $h_4$ is equal to
the real hyperplane $(N_0)\sperp$
of $\PPP_{16}$:
\[
\PPP_{15}=(S_{15}\tensor \R)\cap\PPP_{16}=(N_0)\sperp.
\]
Under the specialization of $Y_{15}$ to $Y_{16}$,
a smooth rational curve on $Y_{15}$ becomes a union of smooth rational curves on $Y_{16}$.
Hence we have
\begin{equation}\label{eq:nef15nef16}
\nef_{15}\supset \PPP_{15}\cap \nef_{16}.
\end{equation}
Composing the embedding $\emb_{15, 16}\colon S_{15}\inj S_{16}$
with Kondo's primitive embedding $\emb_{16}\colon S_{16}\inj \theL$,
we obtain a primitive embedding
\[
\emb_{15}\colon S_{15}\inj \theL.
\]
The orthogonal complement  $R_{15}$ of $S_{15}$ in $\theL$
contains $(-2)$-vectors that form the Dynkin diagram of type $7A_1+A_3$, and hence
$\emb_{15}\colon S_{15}\inj \theL$ satisfies Condition~1 in Section~\ref{subsec:Terminologies}.
The induced chamber
\[
D_{15}:=\emb_{15}\inv (\DDD(\weyl_0))=\emb_{15, 16}\inv(D_{16})
\]
is equal to the outer wall  $D_{16}\cap (N_0)\sperp$ of $D_{16}$,
and hence is contained in 
the nef-and-big cone $\nef_{15}$ by~\eqref{eq:nef15nef16}.
The walls are given in~Table~\ref{table:WallsD15}.
The natural homomorphism from $\OG(S_{15})$ to the automorphism group of
the discriminant form of $S_{15}$ restricted to $\OG(S_{15}, D_{15})$ is
injective, and hence we have
\begin{equation}\label{eq:autX15D15}
\Aut(Y_{15}, D_{15}):=\OG(S_{15}, D_{15})\cap \OG(S_{15})^\omega=\{1\}.
\end{equation}
The vector
\begin{equation}\label{eq:amplefftn}
\ample_{15}:=\pr_S(\weyl_0) = \ample_{16}+\frac{1}{2}N_0\in S_{15}\dual
\end{equation}
satisfies $2 \ample_{15}\in S_{15}$, is of square norm $\intf{\ample_{15}, \ample_{15}}=17/2$,
sits in the interior of $D_{15}$, and is invariant under the action of
$\OG(S_{15}, D_{15})$. In particular,   the class
$2 \ample_{15}$
is an ample class of $Y_{15}$ with square-norm $34$.
This ample class can be used in
the geometric algorithms given in Section~\ref{subsec:GeomCompTools}. 
\par
The group $\OG(S_{15}, D_{15})$ is equal to the stabilizer subgroup of $\alpha_{15}$ in $\OG(S_{15})'$. 
It is immediate to see that
$\OG(S_{15}, D_{15})$ is equal to the subgroup of 
$\OG(S_{16},D_{16})$,
and that 
\[
\OG(S_{15}, D_{15})\cong \mathfrak{S}_6.
\]
The group $\OG(S_{15}, D_{15})\cong \SSSS_6$ decomposes the walls of $D_{15}$ as
in Table~\ref{table:WallsD15}.
In this Table   the column ``up" indicates that,
for example, each wall of $D_{15}$ in the orbit No.~4 is
equal to the face $(N_0)\sperp \cap (v)\sperp$ of $D_{16}$ with codimension $2$,
where $D_{16}\cap (v)\sperp$ is a wall of $D_{16}$
contained in the orbit No.~3 in Table~\ref{table:WallsD16}.
The values $n=\intf{v, v}$, $a=\intf{v, \ample_{15}}$ and
$d=\intf{\ample_{15}, \weyl\sprime_S}$
are also given in Table~\ref{table:WallsD15},
where $v$ is the primitive defining vector of the wall $D_{15}\cap (v)\sperp$,
$D\sprime$ is the induced chamber adjacent to $D_{15}$ across the wall,
$\weyl\sprime$ is
a Weyl vector inducing $D\sprime$,
and $\weyl\sprime_S$ is defined by~\eqref{eq:hDprS}. 
The third column gives a root sublattice of $II_{1,25}$ whose orthogonal complement defines the corresponding wall.
\begin{table}
\[
\begin{array}{ccccccccc}
\textrm{No.}&\textrm{type}&\textrm{Root lattice}&\textrm{size}&\textrm{up}&n&a&d\\
\hline
1&\textrm{outer}&A_3\oplus A_1^{\oplus 8}&10&1&-2&1&19/2\\
2&\textrm{outer}&A_3\oplus A_1^{\oplus 8}&15&1&-2&1&19/2\\
3&\textrm{outer}&D_4\oplus A_3\oplus A_1^{\oplus 4}&15&2&-1/2&5/2&67/2\\
4&\textrm{outer}&D_5\oplus A_1^{\oplus 6}&10&3&-1/2&7/2&115/2\\
\hline
5&\textrm{inner}&A_3\oplus A_2\oplus A_1^{\oplus 6}&6&1&-3/2&3/2&23/2\\
6&\textrm{inner}&A_3^{\oplus 2}\oplus A_1^{\oplus 5}&45&2&-1&2&33/2\\
7&\textrm{inner}&D_4\oplus A_1^{\oplus 7}&6&3&-1&3&53/2\\
8&\textrm{inner}&D_4\oplus A_1^{\oplus 7}&15&3&-1&3&53/2\\
9&\textrm{inner}&A_5\oplus A_1^{\oplus 6}&120&4&-3/4&3&65/2\\
10&\textrm{inner}&D_6\oplus A_1^{\oplus 5}&72&4&-1/4&7/2&213/2
\end{array}
\]
\caption{Walls of $D_{15}$}
\label{table:WallsD15}
\end{table}
\par
In the following, let $O_i$ denote the  orbit of walls of $D_{15}$
under the action of $\OG(S_{15}, D_{15})\cong \SSSS_6$
given in the $i$th row of  Table~\ref{table:WallsD15}.
\subsection{Outer walls of $D_{15}$}\label{subsec:outerwallsD15}
The outer walls of $D_{15}$ are as follows.
\begin{enumerate}[(i)]
\item
The $10$ outer walls in $O_1$ are defined by the classes of
the strict transforms $\sigma(E_{ij})$ of the trope-conics  on $X_{15}$,
where $1\le i<j\le 5$.
\item The $15$ outer walls in $O_2$ are defined by the classes of the
nodal  curves $E_{ij}$ over the nodes of $X_{15}$,
where $1\le i<j\le 6$.
\item Each of the $15$ outer walls in $O_3$ is defined by the class of
a smooth  rational quartic curve $C_{ij,kl,mn}\in |h_4-E_{ij}-E_{kl}-E_{mn}|$ through three nodes of $X_{15}$. This curve remains irreducible under the specialization to $Y_{16}$,
and becomes a smooth rational curve on $Y_{16}$ whose degree with respect to $\ample_{16}$
is $5$.  (See Remark~\ref{rem:smoothrationalcurvesonY16}.)
\item Each of the $10$ outer walls in $O_4$ is defined by the class  of
a smooth rational octic curve
$C_{ijk}\in |2h_4-\sum_{t\ne i,j,k}(E_{it}+E_{jt}+E_{kt})|$.
This curve remains irreducible under the specialization to $Y_{16}$,
and becomes a smooth rational curve on $Y_{16}$
whose degree with respect to $\ample_{16}$
is $7$. (See Remark~\ref{rem:smoothrationalcurvesonY16}.)
\end{enumerate}
We describe these outer walls
combinatorially.
\begin{definition}
For distinct elements $i_1, \dots, i_k$ of $[1,6]=\{1, \dots, 6\}$,
we write by $(i_1 \dots i_k)$ the subset $\{i_1, \dots, i_k\}$  of $\{1, \dots, 6\}$.
Recall that, following Sylvester, a \emph{duad} is a subset $(ij)$ of size $2$ and
a \emph{syntheme} 	 is a non-ordered triple
 $(ij)(kl)(mn)=\{(ij), (kl), (lm)\}$ of duads whose union is $[1,6]$.
We call a \emph{trio}  a subset $(ijk)$ of size $3$ and a \emph{double trio}
a non-ordered pair $(abc)(def)$ of two complementary trios $(abc)$ and $(def)$.
\par
We say that a syntheme $\tau$ is \emph{incident} to a double trio $\theta$
if $|\delta\cap t|=1$ holds for any duad $\delta$ in $\tau$ and
any trio $t$ in $\theta$.
\end{definition}
In view of these terminology,
we have the following indexing of
outer walls,
which is compatible with the action
of $\OG(S_{15}, D_{15})\cong\SSSS_6$ on the set of walls and on 
$[1, 6]$.
%
\begin{enumerate}[(i)]
\item The wall in $O_1$ defined by the curve $\sigma(E_{ij})$
with $1\le i<j\le 5$ is indexed by the double trio $\theta=(ij6)(klm) $ containing
the trio $(ij6)$.
In the following, we write $\sigma(E_\theta)$ for $\sigma(E_{ij})$.
Note that we have
\[
\sigma(E_\theta) \sim
\frac{1}{2}(h_4-(E_{ij}+E_{i6}+E_{j6})-(E_{kl}+E_{lm}+E_{mk})).
\]
 \item The wall in $O_2$
 defined by the nodal curve $E_{ij}$
 with $1\le i<j\le 6$ is indexed by the duad $\delta=(ij)$.
 In the following, we write $E_{\delta}$ for $E_{ij}$.
\item The wall in $O_3$
defined by
the smooth  rational quartic curve $C_{ij, kl, mn}$ is indexed by the syntheme $\tau=(ij)(kl)(mn)$.
\item The wall in $O_4$
defined by
the smooth octic rational curve $C_{ijk}$ is indexed by the double
trio  $(ijk)(lmn)$. Note that we have $C_{ijk}=C_{lmn}$.
\end{enumerate}
In the following,
let $r_1$ be a $(-2)$-vector defining
a wall in $O_1$ corresponding to a double trio $\theta(r_1)=(ijk)(lmn)$,
and let $r_2$ be a $(-2)$-vector defining
a wall in $O_2$ corresponding to a duad $\delta(r_2)=(i\sprime j\sprime)$.
Then we  have
\[
\intf{r_1,  r_2}=\begin{cases}
1 & \textrm{if  $\delta(r_2)$ is a subset of  one of the two trios in  $\theta(r_1)$}, \\
0 & \textrm{otherwise}.
\end{cases}
\]
Since $S_{15}$ is generated by the $10+15$ vectors $r_1$ and $r_2$,
a vector $v\in S_{15}\tensor \Q$ is specified
by $10+15$ numbers $\intf{r_1, v}$
and $\intf{r_2, v}$.
\begin{enumerate}[(i)]
\item[(iii)]
The wall $w_3=D_{15}\cap (v)\sperp$ in  $O_3$
indexed by a syntheme $\tau$
is defined by  the primitive  vector $v\in S_{15}\dual$
that satisfies
\[
\intf{r_1,  v}=\begin{cases}
1 & \textrm{if $\theta(r_1)$ and $\tau$
are incident}, \\
0 & \textrm{otherwise},
\end{cases}
\]
\[
\intf{r_2,  v}=\begin{cases}
1 & \textrm{if  $\delta(r_2)$ is one of the three duads in  $\tau$}, \\
0 & \textrm{otherwise}.
\end{cases}
\]
\item[(iv)]
The wall $w_4=D_{15}\cap (v)\sperp$ in  $O_4$
indexed by a double trio $\theta$
is defined by  the primitive  vector $v\in S_{15}\dual$
that satisfies
\[
\intf{r_1,  v}=\begin{cases}
2 & \textrm{if  $\theta(r_1)=\theta$}, \\
0 & \textrm{otherwise},
\end{cases}
\]
\[
\intf{r_2,  v}=\begin{cases}
0 & \textrm{if $\delta(r_2)$ is a subset of
one of the two trios in $ \theta$}, \\
1 & \textrm{otherwise}.
\end{cases}
\]
\end{enumerate}
\subsection{Indexings of graphs}
For later use, we generalize the notions of duads, synthemes, trios, double trios to the \emph{indexings} of graphs with $6$ vertices.
\begin{definition}
Let $\varGamma$ be a simple graph such that the set $V(\varGamma)$
of vertices is of size $6$.
We denote by $\Indx(\varGamma)$ the set of bijections from $V(\varGamma)$
to $\{1, \dots, 6\}$ modulo
the natural action of the symmetry group
$\textrm{Sym}(\varGamma)$ of the graph $\varGamma$ on $V(\varGamma)$.
\end{definition}
\begin{remark}\label{rem:graph}
Let $E(\varGamma)$ be the set of edges of $\varGamma$.
For a fixed indexing $t\in \Indx(\varGamma)$,
each edge $\{a, b\}$ ($a, b\in V(\varGamma)$) of $\varGamma$ gives a duad
$t(\{a, b\}):=\{t(a), t(b)\}$.
Thus $t$ is considered as a map on $E(\varGamma)$, 
and $t(E(\varGamma))$ can be regarded as a set of nodes of $X_{15}$,
or as a set of nodal curves on $Y_{15}$.
When $t$ runs through the set $\Indx(\varGamma)$,
these sets $t(E(\varGamma))$
form an $\SSSS_6$-orbit of sets of nodes of $X_{15}$.
\end{remark}
\begin{example}
Consider the graphs in Figure~\ref{fig:graphs}.
The set of duads is naturally identified with $\Indx(\varGamma_{\delta})$.
The set of synthemes (resp. double trios) is naturally identified with
$\Indx(\varGamma_{\tau})$ (resp.~with $\Indx(\varGamma_{\theta})$).
\end{example}
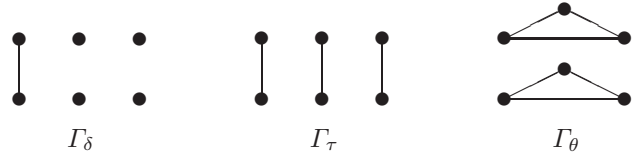
\begin{figure}[hp]
\setlength{\unitlength}{.8cm}
\begin{picture}(3.3, 1.1)(0,-.6)
\put(0,0){\circle*{.2}}
\put(0,1){\circle*{.2}}
\put(1,0){\circle*{.2}}
\put(1,1){\circle*{.2}}
\put(2,0){\circle*{.2}}
\put(2,1){\circle*{.2}}
\put(0,0){\line(0,1){1}}
\put(.8,-.8){\text{$\varGamma_{\delta}$}}
\end{picture}
\quad
\setlength{\unitlength}{.8cm}
\begin{picture}(3.3, 1.1)(0,-.6)
\put(0,0){\circle*{.2}}
\put(0,1){\circle*{.2}}
\put(1,0){\circle*{.2}}
\put(1,1){\circle*{.2}}
\put(2,0){\circle*{.2}}
\put(2,1){\circle*{.2}}
\put(0,0){\line(0,1){1}}
\put(1,0){\line(0,1){1}}
\put(2,0){\line(0,1){1}}
\put(.8,-.8){\text{\hbox{$\varGamma_{\tau}$}}}
\end{picture}
\quad
\setlength{\unitlength}{.8cm}
\begin{picture}(3.3, 2.5)(0,-.6)
\put(0,0){\circle*{.2}}
\put(0,1){\circle*{.2}}
\put(1,.5){\circle*{.2}}
\put(1,1.5){\circle*{.2}}
\put(2,0){\circle*{.2}}
\put(2,1){\circle*{.2}}
\put(0,0){\line(1,0){2}}
\put(0,1){\line(1,0){2}}
\put(0,0){\line(2,1){1}}
\put(0,1){\line(2,1){1}}
\put(2,0){\line(-2,1){1}}
\put(2,1){\line(-2,1){1}}
\put(.8,-.8){\text{\hbox{$\varGamma_{\theta}$}}}
\end{picture}
\caption{Graphs for duads,  synthemes, and double trios}\label{fig:graphs}
\end{figure}
\subsection{Involutions of $Y_{15}$}\label{subsec:involutionsonY15}
In order to exhibit a generating set of $\Aut(Y_{15})$ in Theorem~\ref{thm:15gens},
we present some involutions in $\Aut(Y_{15})$,
and calculate their actions on $S_{15}$.
These involutions are obtained as double-plane
involutions~(see Section~\ref{subsub:doubleplaneinvolutions}).
Recall that we have an interior point $\ample_{15}$  of $D_{15}$
defined by~\eqref{eq:amplefftn}.
The \emph{$\ample_{15}$-degree} of an involution $\invol \in \Aut(Y_{15})$
is defined to be
\[
\deg(\invol):=\intf{\ample_{15}, \ample_{15}^{\invol}}.
\]
\subsubsection{The six involutions $\sigma^{(i)}$}\label{subsubsec:sixsigmas}
Recall that a quintic del Pezzo surface $\mathsf{D}$ is isomorphic to the blow-up of $\PP^2$ at $4$ points no three of which are collinear,
and contains ten $(-1)$-curves that form the Petersen graph
as the dual graph.
\par
As we have already observed, $\mathsf{D}$ can be realized as the quotient of $Y_{15}$ by one of six involutions $\sigma\spar{1}, \dots, \sigma\spar{6}$ coming from a realization of $X_{15}$ as 
the focal surface of a congruence of bi-degree $(2,3)$ isomorphic to $\mathsf{D}$. 
\par
Let $\nu$ be an element of $\{1, \dots, 6\}$.
Then there exist ten duads $\delta_1, \dots, \delta_{10}$ not containing $\nu$.
For $k=1, \dots, 10$,
let $\theta_k$ be the double trio containing
the trio $\{\nu\}\cup \delta_k$.
We denote by $N_k$ the nodal curve $E_{\delta_k}$ corresponding to the duad $\delta_k$,
and by $T_k$ the trope-conic $\sigma(E_{\theta_k})$ corresponding 
the double trio  $\theta_k$.
Note that we have $\intf{N_k, T_k}=1$.
Then we have an involution $\sigma\spar{\nu}$ 
that interchanges $N_k$ and $T_k$ and maps
the class $\etaH$ of a plane section of $X_{15}\subset\PP^3$ to
\[
4 \etaH-\sum_{j\ne \nu} E_{(\nu j)} -2\sum_{k=1}^{10} N_k.
\]
(See~Sections 4.1 and 4.2 of~\cite{PartI}.)
Since the classes of $10+10$ curves $N_k$ and $T_k$
together with $\etaH$ span $S_{15}\tensor\Q$,
the action of $\sigma\spar{\nu}$ on $S_{15}$ is determined uniquely
by these conditions.
There exist five ways to blow down $\mathsf{D}$ to $\PP^2$,
and they correspond to the five choices of $4$-tuples of disjoint $(-1)$-curves.
The composition of these blowings down  
with the double cover $Y_{15}\to \mathsf{D}$ with deck transformation
$\sigma\spar{\nu}$ gives five double-plane covers
$Y_{15}\to \PP^2$.
The branch curve of each of these double-plane covers
 is a $4$-cuspidal plane sextic.
The ten nodal curves $N_k$ and their corresponding trope-conics  $T_k$
 are mapped to $(-1)$-curves on $\mathsf{D}$ 
 which are tangent to the proper transform of the branch curve on $\mathsf{D}$.
The $\ample_{15}$-degree of the involution $\sigma\spar{\nu}$ is $23/2$.
\subsubsection{Reye involutions}\label{subsec:Reye}
See~Section 6.4 of~\cite{PartI}
on the geometric definition of \emph{Reye involution} $\tau_{\Reye}$.
Let $\nu$ be an element of $\{1, \dots, 6\}$.
Then we obtain $10$ nodal curves $N_1, \dots, N_{10}$
as in Section~\ref{subsubsec:sixsigmas}.
We put 
\[
\rrrr:=2\etaH-\sum_{k=1}^{10} N_k,
\]
which is a vector of square-norm $-4$.
The reflection
\[
s_{\rrrr}\colon v\mapsto v+\frac{\intf{v, \rrrr}}{2}\rrrr
\]
is in fact an isometry of $S_{15}$,
and gives the action on $S_{15}$ 
of the Reye involution $\tau_{\Reye}\spar{\nu}$ indexed by $\nu$.
\subsubsection{Double-plane involution associated with a multi-set of nodes}
We consider a class of the form
\[
h_2:=m \etaH -\sum_{\delta} a_{\delta} E_{\delta},
\]
where $\delta$ runs through the set of duads, $m$ is a positive integer,
and $a_{\delta}$ are non-negative integers such that
\[
\intf{h_2, h_2}=4m^2-2\sum_{\delta} a_{\delta}^2=2.
\]
Then $|h_2|$ is the linear system cut out on $X_{15}$ by surfaces of degree $m$
passing through each node $p_{\delta}$
corresponding to the nodal curve $E_{\delta}$ with multiplicity $a_{\delta}$.
We determine whether $h_2$ is nef or not, and if $h_2$ is nef,
determine whether $|h_2|$ is fixed-component free or not,
and if $|h_2|$ is fixed-component free,
calculate the matrix representation of the double-plane involution
associated  with the rational double covering  $Y_{15}\to \PP^2$ induced by $|h_2|$.
\begin{example}\label{example:projectionfromanode}
The class $\etaH-E_{\delta}$ of degree $2$ gives the involution
obtained from the projection  $X_{15}\ratmap \PP^2$
with the center being the node $p_{\delta}$.
The $\mathrm{ADE}$-type of the singularities of the branch curve is $14A_1$.
The $\ample_{15}$-degree of this involution is $53/2$.
\end{example}
\begin{example}\label{example:O9g2}
Consider the graph $\varGamma_7$ with $6$ vertices  in Figure~\ref{figure:7nodalvects}.
The set $\Indx(\varGamma_7)$ is of size $360$.
As was explained in Remark~\ref{rem:graph},
this graph $\varGamma_7$  defines an $\SSSS_6$-orbit of  sets of $7$ nodal curves.
Let
$\{N_1, \dots, N_7\}$ be an element of this orbit.
Then
the class
\[
h_2:=2\etaH-(N_1+\cdots+N_7)
\]
of degree $2$ is nef, and $|h_2|$ defines a rational double covering
$Y_{15}\to \PP^2$, the branch curve of which has singularities
of type $2A_1+3A_3$. The $\ample_{15}$-degree of the associated involution is $81/2$. The rational involution of $X_{15}$ is known as the \emph{Kantor involution} \cite{Coble1919}. It is defined as follows. The net of quadrics through $7$ nodes $x_1,\ldots,x_7$ resolved by the exceptional curves $N_1,\ldots,N_7$ has the additional base point $x_0$ (because three quadrics intersect at $8$ points). For a general point $x\in X_{15}$ the quadrics from the net that vanish at $x$ form a pencil with the base locus a quartic elliptic curve  passing through $x_0,x_1,\ldots,x_7,x$. If we take $x_0$ as the origin in its group law,  the involution sends $x$ to the unique point $x'$ on the curve such that $x+x' = x_0$ in the group law. The Kantor involution is defined when some generality  condition is imposed on the seven points.
It is determined by our choice of $N_1,\ldots,N_7$. 
\begin{figure}
\setlength{\unitlength}{.9cm}
\begin{picture}(3.3, 1.1)(0,0)
\put(0,0){\circle*{.2}}
\put(0,1){\circle*{.2}}
\put(1,0){\circle*{.2}}
\put(1,1){\circle*{.2}}
\put(1.8,.5){\circle*{.2}}
\put(2.8,.5){\circle*{.2}}
\put(0,0){\line(1,0){1}}
\put(0,1){\line(1,0){1}}
\put(0,0){\line(0,1){1}}
\put(1,0){\line(0,1){1}}
\put(1.8,.5){\line(1,0){1}}
\put(1,0){\line(5,3){.7071}}
\put(1,1){\line(5,-3){.7071}}
\end{picture}
\caption{Graph $\varGamma_7$}\label{figure:7nodalvects}
\end{figure}
\end{example}
\begin{example}\label{example:pentagon}
Consider the graph $\varGamma$ with $6$ vertices  in Figure~\ref{figure:pentagon},
which we call a \emph{pentagon}.
The set $\Indx(\varGamma)$ is of size $72$.
This graph $\varGamma$  defines an $\SSSS_6$-orbit of  sets of $5$ nodal curves.
Let
$\{N_1, \dots, N_5\}$ be an element of this orbit.
Then, for $\nu=1, \dots, 5$,
the class
\[
h_{2, \nu}:=3\etaH-2(N_1+\cdots+N_5)+N_{\nu}
\]
of degree $2$ is nef, and $|h_{2, \nu}|$ defines a rational double covering
$Y_{15}\to \PP^2$, the branch curve of which has singularities
of type $6A_1+2D_4$.
The associated involution does not depend on $\nu$, and
its $\ample_{15}$-degree is $213/2$.
The involution is obtained from an admissible pentad $x_1,\ldots,x_5$ of nodes of type 3 from \cite[Table 2]{PartI}. The corresponding birational involution of $X_{15}$ assigns to a general point $x\in X_{15}$  the unique 
remaining intersection point of a rational normal curve through the points $x_1,\ldots,x_5,x$ with the surface. 
\begin{figure}
\setlength{\unitlength}{.9cm}
\begin{picture}(3.3, 2)(0,0)
\put(.55,0){\circle*{.2}}
\put(1.5,0){\circle*{.2}}
\put(0,1.1){\circle*{.2}}
\put(2,1.1){\circle*{.2}}
\put(1,1.6){\circle*{.2}}
\put(1,0.7){\circle*{.2}}
\put(.5,0){\line(1,0){1}}
\put(1.5,0){\line(1,2){.5}}
\put(0,1.1){\line(1,-2){.52}}
\put(0,1.1){\line(2,1){1}}
\put(2,1.1){\line(-2,1){1}}
\put(.1,0){\text{$c$}}
\put(1.7,0){\text{$d$}}
\put(-.4,1){\text{$b$}}
\put(2.2,1){\text{$e$}}
\put(1.2,1.8){\text{$a$}}
\put(1.2,0.8){\text{$f$}}
\end{picture}
\caption{Pentagon}\label{figure:pentagon}
\end{figure}
\end{example}
\subsubsection{Double-plane involution obtained from the model of degree $6$}
\label{subsub:involX6}
In Section 6.2~of~\cite{PartI},
it was shown that the complete linear system $|h_6|$
of the class
\begin{equation}\label{eq:h6}
h_6:=3\etaH- \sum_{\delta}E_{\delta}
\end{equation}
of degree $6$ gives a birational morphism $Y_{15}\to X_{15}\spar{6}$ to a $(2, 3)$-complete intersection
$X_{15}\spar{6}$ in $\PP^4$.
This morphism maps each nodal curve $E_{\delta}$ to a conic and
contracts each trope-conic $\sigma(E_\theta)$ to a node of $X_{15}\spar{6}$.
The image is the intersection of the Segre cubic primal with a quadric.
We have
\begin{equation}\label{eq:h6back}
\etaH=2 h_6- \sum_{\theta}\sigma(E_\theta).
\end{equation}
\begin{remark}
This model $X_{15}\spar{6}$ can be regarded as
the projective dual of $X_{15}\subset \PP^3$ 
(see~\cite[Section~6.4]{PartI}).
Compare~\eqref{eq:h6}~and~\eqref{eq:h6back} with~\eqref{eq:h4h4dual}.
Smooth rational curves defining the walls in the orbits $O_3$ and  $O_4$ are
of degree $6$ with respect to $h_6$.
\end{remark}
Let $\theta$ and $\theta\sprime$ be distinct double trios.
We consider the class
\[
h_2(\theta, \theta\sprime):=h_6-\sigma(E_\theta)-\sigma(E_{\theta\sprime})
\]
of degree $2$.
Then $h_2(\theta, \theta\sprime)$ is nef,
and $|h_2(\theta, \theta\sprime)|$ defines a  double cover
$Y_{15}\to \PP^2$, whose  fiber lies on the  intersection
of $X_{15}\spar{6}$ and  a  plane in $\PP^4$ passing through
the two nodes of $X_{15}\spar{6}$, the images of  $\sigma(E_\theta)$ and $\sigma(E_{\theta\sprime})$
under the map from  $Y_{15}$ to $X_{15}\spar{6}$.
The singularities of the branch curve
of this double-plane covering is of type  $4A_1+2A_3$.
The $\ample_{15}$-degree of the associated involution is $33/2$.
\subsection{Inner walls and the associated extra-automorphisms}
We describe the inner walls of $D_{15}$.
Note that,  by~\eqref{eq:autX15D15},
Condition~4 is satisfied.
We show that Condition~3 is also satisfied by
presenting the automorphism $g_w$ explicitly
for each inner wall $w$ of $D_{15}$.
We call $g_w\in \Aut(Y_{15})$ the \emph{extra-automorphism} for $w\in \innerwalls(D_{15})$.
\par
As in Section~\ref{subsec:outerwallsD15},
let $r_1$ and $r_2$ be the defining $(-2)$-vectors
of outer walls in $O_1$ and $O_2$, respectively,
that is, $r_1$ is the class of the trope-conic $\sigma(E_{\theta(r_1)})$ indexed by
a double trio $\theta(r_1)$, and $r_2$ is the class of 
the nodal curve $E_{\delta(r_2)}$
indexed by
a duad $\delta(r_2)$.
As was remarked in Section~\ref{subsec:outerwallsD15},
a  vector $v\in S_{15}\tensor \Q$ is uniquely characterized
by the $10+15$ numbers $\intf{r_1, v}$
and $\intf{r_2, v}$.
\setcounter{subsubsection}{4}
\subsubsection{The orbit $O_5$}\label{subsub:O5}
Each inner wall $w_5=D_{15}\cap (v)\sperp$ in $O_5$ is indexed by a number
$\nu\in \{1, \dots, 6\}$ in such a way that
the primitive defining vector $v$ of $w_5$ is characterized by
\[
\intf{r_1, v}=0,
\]
\[
\intf{r_2, v}=\begin{cases}
1 & \textrm{if $\nu \in \delta(r_2)$, }\\
0 & \textrm{otherwise}.
\end{cases}
\]
The extra-automorphism  for $w_5$  is the involution $\sigma\spar{\nu}$
defined in Section~\ref{subsubsec:sixsigmas}.
We denote this involution by $\gamma_5(\nu)$.
\subsubsection{The orbit $O_6$}\label{subsub:O6}
Each inner wall $w_6=D_{15}\cap (v)\sperp$ in $O_6$ is indexed by
a non-ordered pair $\{\theta_1, \theta_2\}$ of distinct
double trios in the following way.
Let $\{i_1, i_2\}$ and $\{j_1, j_2\}$ be the two duads
of the form $\tau_1\cap \tau_2$,
where $\tau_1$ is a trio in $\theta_1$ and $\tau_2$ is a trio in $ \theta_2$.
Then we obtain four duads
$\{i_1, j_1\}$, $\{i_1, j_2\}$, $\{i_2, j_1\}$, $\{i_2, j_2\}$.
The primitive defining vector $v$ of $w_6$ is characterized by the following:
\[
\intf{r_1, v}=\begin{cases}
1 & \textrm{if $\theta(r_1)=\theta_1$ or $\theta(r_1)=\theta_2$, }\\
0 & \textrm{otherwise},
\end{cases}
\]
\[
\intf{r_2, v}=\begin{cases}
1 & \textrm{if $\delta(r_2)$ is one of the $4$ duads $\{i_{a},j_{b}\}$ ($a, b=1, 2$), }\\
0 & \textrm{otherwise}.
\end{cases}
\]
The extra-automorphism for $w_6$ is equal to the involution
obtained from $\theta_1, \theta_2$ by the procedure described in 
Section~\ref{subsub:involX6},
which we will denote by
$\gamma_6(\{\theta_1, \theta_2\})$.
\subsubsection{The orbit $O_7$}\label{subsub:O7}
Each inner wall $w_7=D_{15}\cap (v)\sperp$ in $O_7$ is indexed by a number
$\nu\in \{1, \dots, 6\}$ in such a way that
the primitive defining vector $v$ of $w_7$ is characterized by
\[
\intf{r_1, v}=0,
\]
\[
\intf{r_2, v}=\begin{cases}
0 & \textrm{if $\nu \in \delta(r_2)$, }\\
1 & \textrm{otherwise}.
\end{cases}
\]
The extra-automorphism for $w_7$ is equal to the Reye involution
$\tau_{\Reye}\spar{\nu}$ given in 
Section~\ref{subsec:Reye},
which we will denote by
$\gamma_7(\nu)$.
\subsubsection{The orbit $O_8$}\label{subsub:O8}
Each inner wall $w_8=D_{15}\cap (v)\sperp$ in $O_8$ is indexed by a duad $\delta_v$
in such a way that
\[
\intf{r_1, v}=\begin{cases}
0 & \textrm{if $\delta_v$ is a subset of one of  the two trios in $\theta(r_1)$, }\\
1 & \textrm{otherwise},
\end{cases}
\]
\[
\intf{r_2, v}=\begin{cases}
2 & \textrm{if $\delta_v=\delta(r_2)$, }\\
0 & \textrm{otherwise}.
\end{cases}
\]
Let $p_{\delta_v}$ be the node of $X_{15}$ corresponding to
the nodal curve $E_{\delta_v}$ indexed by 
the duad $\delta_v$.
The extra-automorphism for $w_8$ is 
the involution obtained
from the projection $X_{15}\ratmap \PP^2$ with the center  $p_{\delta_v}$,
which we denote by 
 $\gamma_8(\delta_v)$.
\subsubsection{The orbit $O_9$}\label{subsub:O9}
\begin{figure}
\setlength{\unitlength}{.6cm}
\begin{picture}(3.3, 3.7)(0,0)
\put(.5,0){\circle*{.3}}
\put(2,1){\circle*{.3}}
\put(3,1){\circle*{.3}}
\put(4.5,0){\circle*{.3}}
\put(2.5,2){\circle*{.3}}
\put(2.5,3.2){\circle*{.3}}
\put(.5,0){\line(3,2){1.5}}
\put(2,1){\line(1,0){1}}
\put(3,1){\line(3,-2){1.5}}
\put(2,1){\line(1,2){.5}}
\put(3,1){\line(-1,2){.5}}
\put(2.5,2){\line(0,1){1.2}}
\put(0,0){\text{$b$}}
\put(1.4,1){\text{$e$}}
\put(3.3,1){\text{$f$}}
\put(4.8,0){\text{$c$}}
\put(2.8,2){\text{$d$}}
\put(2.8,3.2){\text{$a$}}
\end{picture}
\caption{Tripod}\label{figure:tripod}
\end{figure}
Consider the graph $\varGamma_{\mathrm{tripod}}$ given  in Figure~\ref{figure:tripod}.
Then the set  $\Indx(\varGamma_{\mathrm{tripod}})$ is of size $120$.
Each inner wall $w_9=D_{15}\cap (v)\sperp$ in $O_9$ is indexed by
an indexing $t_v\in \Indx(\varGamma_{\mathrm{tripod}})$ as follows.
For $t\in \Indx(\varGamma_{\mathrm{tripod}})$,
let $\Delta_9(t)$ be the set
$t(E(\varGamma_{\mathrm{tripod}}))$
of duads
obtained from the six edges of $\varGamma_{\mathrm{tripod}}$,
and let $\Theta_9(t)$ be the set of three double trios
obtained by applying $t$ to the following:
\begin{equation}\label{eq:Thetat}
\{\{a, e, f\},\{d, b, c\}\},\;\;
\{\{b,d,f\},\{e, a, c\}\},\;\;
\{\{c,d,e\},\{f,a,b\}\}.
\end{equation}
Then the primitive defining vector $v$ of  $w_9$ is characterized by
\[
\intf{r_1, v}=\begin{cases}
1 & \textrm{if $\theta(r_1)\in \Theta_9(t_v)$, }\\
0 & \textrm{otherwise},
\end{cases}
\]
\[
\intf{r_2, v}=\begin{cases}
1 & \textrm{if $\delta(r_2)\in \Delta_9(t_v)$, }\\
0 & \textrm{otherwise}.
\end{cases}
\]
The extra-automorphism  for $w_9$ is
of infinite order,
and can be written as a product of two involutions as follows.
Let $\Theta_9(t_v)$ be $\{\theta_1, \theta_2, \theta_3\}$,
and for $\theta_i\in \Theta_9(t_v)$, let $\{\theta_j, \theta_k\}$
be the complement of $\{\theta_i\}$ in $\Theta_9(t_v)$.
Let $\tau_1$ and $\tau_2$ be the complementary trios such that
$\theta_i=\{\tau_1, \tau_2\}$.
Interchanging $\tau_1$ and $\tau_2$ if necessary,
we can assume that
$\tau_1\cap \{t_v(a), t_v(b), t_v(c)\}$ is a duad $\delta_i$.
Connecting the two vertices of $\varGamma_{\mathrm{tripod}}$
that are mapped by $t_v$ to $\delta_i$,
we obtain the graph $\varGamma_7$ in Figure~\ref{figure:7nodalvects}
with an indexing $\tilde{t}_v\in \Indx(\varGamma_7)$ induced by
$t_v\in \Indx(\varGamma_{\mathrm{tripod}})$.
Then,  by Example~\ref{example:O9g2},
we obtain  a double-plane  involution,
which will be denoted by  $\gamma_9\sprime(t_v, \theta_i)$.
It turns out that
\[
\gamma_9(t_v)=\gamma_6(\{\theta_j, \theta_k\}) \gamma_9\sprime(t_v, \theta_i)
\]
for $i=1, 2,3$, where $\gamma_6(\{\theta_j, \theta_k\})$ is defined in
Section~\ref{subsub:O6},
is independent of $i$ and is equal to the extra-automorphism for $w_9$.
\subsubsection{The orbit $O_{10}$}\label{subsub:O10}
Consider the graph $\varGamma_{\mathrm{penta}}$ given in Figure~\ref{figure:pentagon}.
Each inner wall $w_{10}=D_{15}\cap (v)\sperp$ in $O_{10}$ is indexed by
an indexing $p_v\in \Indx(\varGamma_{\mathrm{penta}})$ as follows.
For $p\in \Indx(\varGamma_{\mathrm{penta}})$,
let $\Delta_{10}(p)$ be the set
$t(E(\varGamma_{\mathrm{penta}}))$
of duads
obtained from the five edges of $\varGamma_{\mathrm{penta}}$,
and let $\Theta_{10}(p)$ be the set of five double trios
obtained by applying $p$ to the following:
\begin{eqnarray*}
&&\{\{a,c,d\},\{b,e,f\}\},\;\;
\{\{b,d,e\},\{a,e,f\}\},\;\;
\{\{c, e, a\},\{b,d, f\}\},\;\;\\
&&\{\{d, a, b\},\{c, e, f\}\},\;\;
\{\{e, b, c\},\{d, a, f\}\}.
\end{eqnarray*}
Then the primitive defining vector $v$ of  $w_{10}$ is characterized by
\[
\intf{r_1, v}=\begin{cases}
1 & \textrm{if $\theta(r_1)\in \Theta_{10}(p_v)$, }\\
0 & \textrm{otherwise},
\end{cases}
\]
\[
\intf{r_2, v}=\begin{cases}
1 & \textrm{if $\delta(r_2)\in \Delta_{10}(p_v)$, }\\
0 & \textrm{otherwise}.
\end{cases}
\]
The extra-automorphism 
for  $w_{10}$ is the involution given in Example~\ref{example:pentagon},
which we will denote by $\gamma_{10} (p_v)$.
\subsection{Main Theorems}\label{subsec:maintheorems}
Summarizing the results in the previous section,
we obtain the following:
\begin{theorem}\label{thm:15gens}
The automorphism group $\Aut(Y_{15})$ is generated by the following elements.
\begin{enumerate}[{\rm (i)}]
\setcounter{enumi}{4}
\item The six involutions $\gamma_5(\nu)$, where $\nu=1, \dots, 6$,
that make $Y_{15}$
 the focal surface of a congruence of bi-degree $(2,3)$.
\item The $45$ involutions $\gamma_6(\{\theta_1, \theta_2\})$,
where $\{\theta_1, \theta_2\}$ is a non-ordered pair
of double trios.
Each $\gamma_6(\{\theta_1, \theta_2\})$
is obtained from the projection
$X_{15}\spar{6}\ratmap \PP^2$,
where $X_{15}\spar{6}$ is
the $(2, 3)$-complete intersection model  of  $Y_{15}$
given by the class~\eqref{eq:h6}, and
the center of the projection  is  the line passing through the two nodes of $X_{15}\spar{6}$
corresponding to $\theta_1$ and $\theta_2$.
\item
The six Reye involutions $\gamma_7(\nu)$, where $\nu=1, \dots, 6$.

\item
The $15$ involutions $\gamma_8(\delta)$
obtained from the projection of $X_{15}\ratmap\PP^2$
with the center being the node of $X_{15}$
corresponding to the duad $\delta$.
\item
The $120$ automorphisms
\[
\gamma_{9}(t)=\gamma_{6}(\{\theta_j, \theta_k\}) \gamma_{9}\sprime(t, \theta_i)
\]
of infinite order,
where $t\in \Indx(\varGamma_{\mathrm{tripod}})$,
$\Theta_9(t)=\{\theta_i, \theta_j, \theta_k\}$ is
the set of three double trios obtained by
applying  $t$ to~\eqref{eq:Thetat},
$\gamma_{6}(\{\theta_j, \theta_k\})$
is the involution defined in {\rm (vi)},
and $\gamma_{9}\sprime(t, \theta_i)$
is the involution
given in Example~\ref{example:O9g2}, where
the graph defining the $7$ nodes in Figure~\ref{figure:7nodalvects}
is  obtained from
$\varGamma_{\mathrm{tripod}}$ with indexing $t$
by connecting  the  vertices that are the intersection of 
$\{t(a), t(b), t(c)\}$  and  one of the two trios in $\theta_i$.
\item
The $72$ involutions $\gamma_{10}(p)$
given  in Example~\ref{example:pentagon},
where $p\in \Indx(\varGamma_{\mathrm{penta}})$.
\end{enumerate}
\qed
\end{theorem}
%
%
%
We then describe the defining relations of $\Aut(Y_{15})$
with respect to the generating set
\begin{equation}\label{eq:thegenerators}
\gamma_5(\nu),\;\;
\gamma_6(\{\theta_1, \theta_2\}), \;\;
\gamma_7(\nu), \;\;
\gamma_8(\delta), \;\;
\gamma_9(t), \;\;
\gamma_{10}(p), \;\;
\end{equation}
given in Theorem~\ref{thm:15gens}.
By Theorem~\ref{thm:BorcherdsRels},
it is enough to calculate the relations $\RRR_1$ and $\RRR_2$
defined by~\eqref{eq:RRR1} and~\eqref{eq:RRR2}, respectively.
The method to calculate $\RRR_2$ is explained in Sections~\ref{subsec:algofaces}~and~\ref{subsec:algorels}.
\begin{theorem}\label{thm:15rels}
The relations in $\RRR_1$ are as follows:
\[
\gamma_5(\nu)^2=1, \;\;
\gamma_6(\{\theta_1, \theta_2\})^2=1, \;\;
\gamma_7(\nu)^2=1, \;\;
\gamma_8(\delta)^2=1, \;\;
\gamma_{10}(p)^2=1, \;\;
\]
and
\[
\gamma_9(t) \gamma_9(t\sprime)=1, \;\;
\]
where
$t, t\sprime\in \Indx(\varGamma_{\mathrm{tripod}})$
are pairs of distinct indexings 
such that $\Theta_9(t)=\Theta_9(t\sprime)$.
\par
There exist exactly $5235$ inner faces of $D_{15}$ with codimension $2$,
and they are decomposed into $19$ orbits $F_1, \dots, F_{19}$
under the action  of $\OG(S_{15}, D_{15})\cong \SSSS_6$.
In the table below, the {\rm size} is the number
of faces in the orbit $F_i$,
$f$ is an element of $F_i$,
and $R_f=(g_m, \dots, g_1)$ is the
relation $g_m\cdots g_1=1$
associated with the simple chamber loop $(D_0, \dots, D_m)$
around $(f, D_0)$,
where the starting chamber $D_0$ is 
our induced chamber $D_{15}$.
\par
\medskip
{
\rm
%
$F_{1}$: ${\rm size}=180$ \par 
$f=w_{5}(1)\cap w_{6}(\{(123), (124)\})$\par
$R_f=(\gamma_{6}(\{(123), (124)\}), \gamma_{5}(2), \gamma_{6}(\{(125), (126)\}), \gamma_{5}(1))$ \par
\hrulefill \par
$F_{2}$: ${\rm size}=30$ \par 
$f=w_{5}(1)\cap w_{7}(2)$\par
$R_f=(\gamma_{7}(2), \gamma_{5}(1), \gamma_{8}((12)), \gamma_{5}(1))$ \par
\hrulefill \par
$F_{3}$: ${\rm size}=30$ \par 
$f=w_{5}(1)\cap w_{8}((12))$\par
$R_f=(\gamma_{8}((12)), \gamma_{5}(1), \gamma_{7}(2), \gamma_{5}(1))$ \par
\hrulefill \par
$F_{4}$: ${\rm size}=360$ \par 
$f=w_{5}(1)\cap w_{9}([234156])$\par
$R_f=(\gamma_{9}([156234]), \gamma_{5}(2), \gamma_{9}([256143]), \gamma_{5}(1))$ \par
\hrulefill \par
$F_{5}$: ${\rm size}=45$ \par 
$f=w_{6}(\{(123), (124)\})\cap w_{6}(\{(135), (145)\})$\par
$R_f=(\gamma_{6}(\{(135), (145)\}), \gamma_{6}(\{(123), (124)\}), \gamma_{6}(\{(135), (145)\}), \gamma_{6}(\{(123), (124)\}))$ \par
\hrulefill \par
$F_{6}$: ${\rm size}=180$ \par 
$f=w_{6}(\{(123), (124)\})\cap w_{6}(\{(123), (125)\})$\par
$R_f=(\gamma_{6}(\{(123), (125)\}), \gamma_{6}(\{(125), (126)\}), \gamma_{6}(\{(124), (126)\}), \gamma_{6}(\{(123), (124)\}))$ \par
\hrulefill \par
$F_{7}$: ${\rm size}=180$ \par 
$f=w_{6}(\{(123), (124)\})\cap w_{6}(\{(123), (135)\})$\par
$R_f=(\gamma_{6}(\{(123), (135)\}), \gamma_{6}(\{(124), (135)\}), \gamma_{6}(\{(123), (124)\}), $\par
\hskip 1.5cm $\gamma_{6}(\{(123), (135)\}), \gamma_{6}(\{(124), (135)\}), \gamma_{6}(\{(123), (124)\}))$ \par
\hrulefill \par
$F_{8}$: ${\rm size}=360$ \par 
$f=w_{6}(\{(123), (124)\})\cap w_{6}(\{(125), (135)\})$\par
$R_f=(\gamma_{6}(\{(125), (135)\}), \gamma_{6}(\{(124), (146)\}), \gamma_{6}(\{(125), (126)\}), $\par
\hskip 1.5cm $\gamma_{6}(\{(123), (146)\}), \gamma_{6}(\{(126), (135)\}), \gamma_{6}(\{(123), (124)\}))$ \par
\hrulefill \par
$F_{9}$: ${\rm size}=90$ \par 
$f=w_{6}(\{(123), (124)\})\cap w_{7}(3)$\par
$R_f=(\gamma_{7}(3), \gamma_{6}(\{(123), (124)\}), \gamma_{7}(4), \gamma_{6}(\{(123), (124)\}))$ \par
\hrulefill \par
$F_{10}$: ${\rm size}=180$ \par 
$f=w_{6}(\{(123), (124)\})\cap w_{8}((15))$\par
$R_f=(\gamma_{8}((15)), \gamma_{6}(\{(123), (124)\}), \gamma_{8}((26)), \gamma_{6}(\{(123), (124)\}))$ \par
\hrulefill \par
$F_{11}$: ${\rm size}=360$ \par 
$f=w_{6}(\{(123), (124)\})\cap w_{9}([135642])$\par
$R_f=(\gamma_{9}([246531]), \gamma_{6}(\{(123), (124)\}), \gamma_{9}([246531]), \gamma_{6}(\{(123), (124)\}))$ \par
\hrulefill \par
$F_{12}$: ${\rm size}=360$ \par 
$f=w_{6}(\{(123), (124)\})\cap w_{9}([134256])$\par
$R_f=(\gamma_{9}([256134]), \gamma_{9}([134265]), \gamma_{6}(\{(123), (124)\}), $\par
\hskip 1.5cm $\gamma_{9}([156243]), \gamma_{9}([234156]), \gamma_{6}(\{(123), (124)\}))$ \par
\hrulefill \par
$F_{13}$: ${\rm size}=360$ \par 
$f=w_{6}(\{(123), (124)\})\cap w_{9}([123564])$\par
$R_f=(\gamma_{9}([456312]), \gamma_{9}([124563]), \gamma_{6}(\{(123), (124)\}), $\par
\hskip 1.5cm $\gamma_{9}([456312]), \gamma_{9}([124563]), \gamma_{6}(\{(123), (124)\}))$ \par
\hrulefill \par
$F_{14}$: ${\rm size}=720$ \par 
$f=w_{6}(\{(123), (124)\})\cap w_{9}([134526])$\par
$R_f=(\gamma_{9}([256314]), \gamma_{6}(\{(125), (126)\}), \gamma_{9}([234651]), \gamma_{6}(\{(123), (124)\}))$ \par
\hrulefill \par
$F_{15}$: ${\rm size}=720$ \par 
$f=w_{6}(\{(123), (124)\})\cap w_{9}([135462])$\par
$R_f=(\gamma_{9}([246513]), \gamma_{9}([235461]), \gamma_{6}(\{(123), (124)\}), $\par
\hskip 1.5cm $\gamma_{9}([235614]), \gamma_{9}([246351]), \gamma_{6}(\{(123), (124)\}))$ \par
\hrulefill \par
$F_{16}$: ${\rm size}=360$ \par 
$f=w_{6}(\{(123), (124)\})\cap w_{10}([13526])$\par
$R_f=(\gamma_{10}([13526]), \gamma_{6}(\{(123), (124)\}), \gamma_{10}([15246]), \gamma_{6}(\{(123), (124)\}))$ \par
\hrulefill \par
$F_{17}$: ${\rm size}=180$ \par 
$f=w_{9}([123456])\cap w_{9}([123465])$\par
$R_f=(\gamma_{9}([456123]), \gamma_{6}(\{(145), (146)\}), \gamma_{9}([234561]), $\par
\hskip 1.5cm $\gamma_{9}([156432]), \gamma_{6}(\{(145), (146)\}), \gamma_{9}([123465]))$ \par
\hrulefill \par
$F_{18}$: ${\rm size}=180$ \par 
$f=w_{9}([123456])\cap w_{9}([126453])$\par
$R_f=(\gamma_{9}([456123]), \gamma_{6}(\{(123), (126)\}), \gamma_{9}([126453]), $\par
\hskip 1.5cm $\gamma_{9}([456123]), \gamma_{6}(\{(123), (126)\}), \gamma_{9}([126453]))$ \par
\hrulefill \par
$F_{19}$: ${\rm size}=360$ \par 
$f=w_{9}([123456])\cap w_{9}([124356])$\par
$R_f=(\gamma_{9}([356124]), \gamma_{6}(\{(126), (134)\}), \gamma_{9}([356241]), $\par
\hskip 1.5cm $\gamma_{9}([123645]), \gamma_{6}(\{(126), (134)\}), \gamma_{9}([123456]))$ \par
\hrulefill \par

}
\qed
\end{theorem}
In the table in Theorem~\ref{thm:15rels},
a wall $w_i(K)$ is the wall in the orbit $O_i$ indexed by $K$,
where
\begin{itemize}
\item $K$ is the number $\nu\in \{1, \dots, 6\}$
that indexes $w_i(K)$ when $i=5$ or $7$,
\item $K$ is the pair of trios $(i_1 j_1 k_1)\in \theta_1$
and $(i_2 j_2 k_2)\in \theta_2$
 when $i=6$ and $w_i(K)$ is indexed by $\{\theta_1, \theta_2\}$
 (the other trio of $\theta_{\nu}$ is $\{1, \dots, 6\}\setminus (i_{\nu} j_{\nu} k_{\nu})$),
 \item $K$ is the duad $\delta=(ij)$ that indexes $w_i(K)$
  when $i=8$,
  \item $K$ is  $[t(a)t(b)t(c)t(d)t(e)t(f)]$
   when $i=9$ and  $w_i(K)$ is indexed by $t\in \Indx(\varGamma_{\mathrm{tripod}})$,
   and
   \item
   $K$ is  $[p(a)p(b)p(c)p(d)p(e)]$
    when $i=10$ and  $w_i(K)$ is indexed by $p\in \Indx(\varGamma_{\mathrm{penta}})$.
\end{itemize}
The automorphism $\gamma_i (K)$
is the extra-automorphism associated with the wall $w_i(K)$.
\begin{example}
By the relations associated with the faces $F_2$ and $F_3$,
we obtain the following relations 
\[
\gamma_7(\nu)=\gamma_5(\mu) \gamma_8(\{\nu, \mu\}) \gamma_5(\mu)
\]
for $\mu, \nu \in [1, 6]$ with $\mu\ne \nu$
between the Reye involution $\gamma_7(\nu)$,
the involution $\gamma_8(\{\nu, \mu\})$ induced by the projection
$X_{15}\ratmap \P^2$ from the node $p_{\{\nu, \mu\}}$,
and the involution $\gamma_5(\mu)$ obtained 
from the double covering of a quintic del Pezzo surface $\mathsf{D}$
in the Grassmannian.
\end{example}
\subsection{Writing an automorphism as a product of the generators}

Suppose that an automorphism $g\in \Aut(Y_{15})$ is given.
Then $g$ is expressed as a product of the generators~\eqref{eq:thegenerators}.
To obtain such an expression, we can use the following method.
We choose a vector $\alpha\sprime \in S_{15}$
from the interior of $D_{15}$.
By making the choice random enough,
we can assume that the line segment $\ell$ 
in $\PPP_{15}$ connecting $\alpha\sprime$
and $\alpha^{\prime g}$ 
does not intersect any face of induced chambers with codimension $\ge 2$.
Let $D\sprime$  be the induced chamber
containing 
$\alpha^{\prime g}$ in its interior.
We calculate the set of $(-2)$-vectors $r_1, \dots, r_N$ of $\theL$
such that the hyperplane $(r_i)\sperp$ of $\PPP(\theL)$ separates 
the points $\emb_{15}(\alpha\sprime)$ and $\emb_{15}(\alpha^{\prime g})$.
Calculating the vector $\pr_S(r_i)$,
where $\pr_S\colon \theL\to S_{15}\dual$ is the natural projection,
we obtain the list of walls of induced chambers that 
the line segment $\ell$ intersects.
Thus we obtain a chamber path $(D\spar{0}, \dots, D\spar{m})$
from $D\spar{0}=D_{15}$ to  $D\spar{m}=D\sprime$  such that
$\ell$ intersects every $D\spar{i}$.
By the method described in Section~\ref{subsec:BorcherdsRels},
we calculate the sequence $g_1, \dots, g_m\in \Gen$ such that
\[
D\spar{i}={D_{15}}^{g_i \dots g_1}
\]
for $i=1, \dots, m$.
Thus we obtain an expression $g=g_m \cdots g_1$.
\begin{example}
We apply this method to the involutions obtained from the admissible pentads
(see~\cite[Section~9]{PartI}).
Under the action of $\SSSS_6$, the admissible pentads are decomposed 
into $9$ orbits as in~\cite[Table 2]{PartI}.
In Table~\ref{table:pentads},
we choose an admissible pentad from each orbit,
and write the associated involution as a product of generators.
The $5$ nodes are given by corresponding duads.
The involution is the product of 
generators in the second line.
\begin{table}
{\small 
\renewcommand{\arraystretch}{1.1}
\[
\begin{array}{l}
\mathrm{type}=1, \;\;\mathrm{nodes}=\{(12), (13), (14), (15), (16)\}:\\
\gamma\sb{5} (1), \gamma\sb{7} (1), \gamma\sb{5} (1)\mystrutd{10pt}\\
\hline
\mathrm{type}=2, \;\;\mathrm{nodes}=\{(12), (13), (24), (34), (56)\}:\\
\gamma\sb{6} (\{(145), (146)\}), \gamma\sb{8} ((56)), \gamma\sb{6} (\{(145), (146)\})\mystrutd{10pt}\\
\hline
\mathrm{type}=3, \;\;\mathrm{nodes}=\{(12), (13), (24), (35), (45)\}:\\
\gamma\sb{10} ([12453])\mystrutd{10pt}\\
\hline
\mathrm{type}=4, \;\;\mathrm{nodes}=\{(12), (13), (14), (23), (24)\}:\\
\gamma\sb{6} (\{(125), (126)\}), \gamma\sb{8} ((12)), \gamma\sb{6} (\{(125), (126)\})\mystrutd{10pt}\\
\hline
\mathrm{type}=5, \;\;\mathrm{nodes}=\{(12), (13), (14), (15), (26)\}:\\
\gamma\sb{5} (1), \gamma\sb{5} (6), \gamma\sb{8} ((12)), \gamma\sb{5} (6), \gamma\sb{5} (1)\mystrutd{10pt}\\
\hline
\mathrm{type}=6, \;\;\mathrm{nodes}=\{(12), (13), (14), (15), (23)\}:\\
\gamma\sb{5} (1), \gamma\sb{5} (6), \gamma\sb{8} ((45)), \gamma\sb{5} (6), \gamma\sb{5} (1)\mystrutd{10pt}\\
\hline
\mathrm{type}=7, \;\;\mathrm{nodes}=\{(12), (13), (14), (25), (35)\}:\\
\gamma\sb{6} (\{(145), (156)\}), \gamma\sb{8} ((56)), \gamma\sb{6} (\{(145), (156)\})\mystrutd{10pt}\\
\hline
\mathrm{type}=8, \;\;\mathrm{nodes}=\{(12), (13), (14), (23), (25)\}:\\
\gamma\sb{9} ([123456]), \gamma\sb{8} ((45)), \gamma\sb{9} ([456123])\mystrutd{10pt}\\
\hline
\mathrm{type}=9, \;\;\mathrm{nodes}=\{(12), (13), (14), (25), (36)\}:\\
\gamma\sb{9} ([123456]), \gamma\sb{8} ((14)), \gamma\sb{9} ([456123])\mystrutd{10pt}\\
\hline
\end{array}
\]
}
\caption{Admissible pentads}\label{table:pentads}
\end{table}
\end{example}
%

\begin{thebibliography}{10}

\bibitem{Baker1925IV}
H.~F. Baker.
\newblock {\em Principles of geometry. {V}olume 4. {H}igher geometry}.
\newblock Cambridge Library Collection. Cambridge University Press, Cambridge,
  2010.
\newblock Reprint of the 1925 original.

\bibitem{Bor1}
Richard Borcherds.
\newblock Automorphism groups of {L}orentzian lattices.
\newblock {\em J. Algebra}, 111(1):133--153, 1987.

\bibitem{Bor2}
Richard~E. Borcherds.
\newblock Coxeter groups, {L}orentzian lattices, and {$K3$} surfaces.
\newblock {\em Internat. Math. Res. Notices}, 1998(19):1011--1031, 1998.

\bibitem{Coble1919}
Arthur Coble.
\newblock The ten nodes of the rational sextic and of the Cayley symmetroid.
\newblock {\em Amer. J. Math.}, 41(4):243--265, 1919.

\bibitem{Conway1983}
J.~H. Conway.
\newblock The automorphism group of the {$26$}-dimensional even unimodular
  {L}orentzian lattice.
\newblock {\em J. Algebra}, 80(1):159--163, 1983.

\bibitem{CSBook}
J.~H. Conway and N.~J.~A. Sloane.
\newblock {\em Sphere packings, lattices and groups}, volume 290 of {\em
  Grundlehren der Mathematischen Wissenschaften}.
\newblock Springer-Verlag, New York, third edition, 1999.

\bibitem{PartI}
Igor Dolgachev.
\newblock 15-nodal quartic surfaces. {P}art {I}: {Q}uintic del {P}ezzo surfaces
  and congruences of lines in {${\mathbf P}^3$}, 2019.
\newblock Preprint, ar{X}iv:1906.12295.

\bibitem{CAG}
Igor~V. Dolgachev.
\newblock {\em Classical algebraic geometry. A modern view}.
\newblock Cambridge University Press, Cambridge, 2012.

\bibitem{Hudson1905}
R.~W. H.~T. Hudson.
\newblock {\em Kummer's quartic surface}.
\newblock Cambridge Mathematical Library. Cambridge University Press,
  Cambridge, 1990.
\newblock With a foreword by W. Barth, Revised reprint of the 1905 original.

\bibitem{Hutchinson1899}
J.~I. Hutchinson.
\newblock The {H}essian of the cubic surface.
\newblock {\em Bull. Amer. Math. Soc.}, 5(6):282--292, 1899.

\bibitem{Hutchinson1900}
J.~I. Hutchinson.
\newblock The {H}essian of the cubic surface. {II}.
\newblock {\em Bull. Amer. Math. Soc.}, 6(8):328--337, 1900.

\bibitem{Hutchinson1901}
J.~I. Hutchinson.
\newblock On some birational transformations of the {K}ummer surface into
  itself.
\newblock {\em Bull. Amer. Math. Soc.}, 7(5):211--217, 1901.

\bibitem{Keum1997}
Jong~Hae Keum.
\newblock Automorphisms of {J}acobian {K}ummer surfaces.
\newblock {\em Compositio Math.}, 107(3):269--288, 1997.

\bibitem{Kondo1998}
Shigeyuki Kondo.
\newblock The automorphism group of a generic {J}acobian {K}ummer surface.
\newblock {\em J. Algebraic Geom.}, 7(3):589--609, 1998.

\bibitem{Nikulin1979}
V.~V. Nikulin.
\newblock Integer symmetric bilinear forms and some of their geometric
  applications.
\newblock {\em Izv. Akad. Nauk SSSR Ser. Mat.}, 43(1):111--177, 238, 1979.
\newblock English translation: Math USSR-Izv. 14 (1979), no. 1, 103--167
  (1980).

\bibitem{Nikulin1991}
Viacheslav~V. Nikulin.
\newblock Weil linear systems on singular {$K3$} surfaces.
\newblock In {\em Algebraic geometry and analytic geometry ({T}okyo, 1990)},
  ICM-90 Satell. Conf. Proc., pages 138--164. Springer, Tokyo, 1991.

\bibitem{Ohashi2009}
Hisanori Ohashi.
\newblock Enriques surfaces covered by {J}acobian {K}ummer surfaces.
\newblock {\em Nagoya Math. J.}, 195:165--186, 2009.

\bibitem{PSS1971}
I.~I. Pjatecki{\u\i}-{\v{S}}apiro and I.~R. {\v{S}}afarevi{\v{c}}.
\newblock Torelli's theorem for algebraic surfaces of type {${\rm K}3$}.
\newblock {\em Izv. Akad. Nauk SSSR Ser. Mat.}, 35:530--572, 1971.
\newblock Reprinted in I. R. Shafarevich, Collected Mathematical Papers,
  Springer-Verlag, Berlin, 1989, pp.~516--557.

\bibitem{SaintDonat1974}
B.~Saint-Donat.
\newblock Projective models of {$K-3$} surfaces.
\newblock {\em Amer. J. Math.}, 96:602--639, 1974.

\bibitem{ShimadaChar5}
Ichiro Shimada.
\newblock Projective models of the supersingular {$K3$} surface with {A}rtin
  invariant 1 in characteristic 5.
\newblock {\em J. Algebra}, 403:273--299, 2014.

\bibitem{Shimada2015}
Ichiro Shimada.
\newblock An algorithm to compute automorphism groups of {$K3$} surfaces and an
  application to singular {$K3$} surfaces.
\newblock {\em Int. Math. Res. Not. IMRN}, (22):11961--12014, 2015.

\bibitem{AutEMS}
Ichiro Shimada.
\newblock The elliptic modular surface of level 4 and its reduction modulo 3,
  2018.
\newblock Preprint. ar{X}iv:1806.05787.

\bibitem{compdata15nodal}
Ichiro Shimada.
\newblock 15-nodal quartic surfaces. {P}art {II}: {T}he automorphism group:
  {C}omputational data, 2019.
\newblock 
  \url{http://www.math.sci.hiroshima-u.ac.jp/~shimada/K3andEnriques.html}.

\bibitem{ShimadaVeniani2019}
Ichiro Shimada and Davide~Cesare Veniani.
\newblock {E}nriques involutions on singular {$K3$} surfaces of small
  discriminants.
\newblock {P}reprint, ar{X}iv:1902.00229, 2019.
\newblock To appear in {A}nn. {S}c. {N}orm. {S}uper. {P}isa {C}l. {S}ci.

\bibitem{Shioda1977}
Tetsuji Shioda.
\newblock Some results on unirationality of algebraic surfaces.
\newblock {\em Math. Ann.}, 230(2):153--168, 1977.

\bibitem{GAP}
{The GAP Group}.
\newblock {\em {G}{A}{P} - {G}roups, {A}lgorithms, and {P}rogramming}.
\newblock Version 4.8.6; 2016 (\url{http://www.gap-system.org}).

\bibitem{VinbergShvartsman}
\`E.~B. Vinberg and O.~V. Shvartsman.
\newblock Discrete groups of motions of spaces of constant curvature.
\newblock In {\em Geometry, {II}}, volume~29 of {\em Encyclopaedia Math. Sci.},
  pages 139--248. Springer, Berlin, 1993.

\end{thebibliography}
%

\end{document}